\documentclass[11pt,a4paper]{article}
\usepackage{bbm}
\usepackage{mathrsfs}
\usepackage{amsfonts}
\usepackage{graphicx,indentfirst}
\usepackage{subfigure}
\usepackage{amsmath,amssymb,amsthm}
\usepackage{natbib}
\usepackage{epstopdf}
\usepackage{fullpage}
\usepackage{multirow}
\usepackage{booktabs}
\usepackage{color}
\usepackage[colorlinks,citecolor=blue,urlcolor=blue]{hyperref}

\allowdisplaybreaks[4]
\numberwithin{equation}{section}
\theoremstyle{plain}
\newtheorem{thm}{Theorem}[section]
\newtheorem{lemma}{Lemma}[section]
\newtheorem{proposition}{Proposition}[section]
\newtheorem{corollary}{Corollary}[section]
\newtheorem{assumption}{Assumption}
\newtheorem*{assumption'}{Assumption $3'$}
\newtheorem{example}{Example}
\theoremstyle{definition}
\newtheorem{definition}{Definition}[section]
\theoremstyle{remark}
\newtheorem{remark}{Remark}[section]

\numberwithin{equation}{section}	
\usepackage{mathtools}	% need for `show only references'
\mathtoolsset{showonlyrefs=true}
\allowdisplaybreaks[4]

\newcommand{\gA}{\mathcal{A} }
\newcommand{\gG}{\mathcal{G} }

\newcommand{\gP}{\mathcal{P} }

\newcommand{\gN}{\mathcal{N} }
\newcommand{\gQ}{\mathcal{Q} }
\newcommand{\gL}{\mathcal{L} }
\newcommand{\gD}{\mathcal{D} }
\newcommand{\gF}{\mathcal{F} }

\makeatletter

\newcommand{\Rmnum}[1]{\expandafter\@slowromancap\romannumeral #1@}
\makeatother

\author{Weiwei Guo\thanks{Department of Systems Engineering and Engineering Management, The Chinese University of Hong Kong, Shatin, Hong Kong. Email: wwguo@se.cuhk.edu.hk. }
\and
Lingfei Li\thanks{Department of Systems Engineering and Engineering Management, The Chinese University of Hong Kong, Shatin, Hong Kong. Email: lfli@se.cuhk.edu.hk. Corresponding author.}}

\begin{document}
\title{Parametric Inference for Discretely Observed Subordinate Diffusions\footnote{This research was supported by Hong Kong Research Grant Council GRF Grant No. 14205816.}}
\maketitle

\begin{abstract}
Subordinate diffusions are constructed by time changing diffusion processes with an independent L\'{e}vy subordinator. This is a rich family of Markovian jump processes which exhibit a variety of jump behavior and have found many applications. This paper studies parametric inference of discretely observed ergodic subordinate diffusions. We solve the identifiability problem for these processes using spectral theory and propose a two-step estimation procedure based on estimating functions. In the first step, we use an estimating function that only involves diffusion parameters. In the second step, a martingale estimating function based on eigenvalues and eigenfunctions of the subordinate diffusion is used to estimate the parameters of the L\'{e}vy subordinator and the problem of how to choose the weighting matrix is solved. When the eigenpairs do not have analytical expressions, we apply the constant perturbation method with high order corrections to calculate them numerically and the martingale estimating function can be computed efficiently. Consistency and asymptotic normality of our estimator are established considering the effect of numerical approximation. Through numerical examples, we show that our method is both computationally and statistically efficient. A subordinate diffusion model for VIX (CBOE volatility index) is developed which provides good fit to the data.
\end{abstract}

\bigskip
\hspace{0.3cm}\emph{Keywords}: diffusions, time change, subordinate diffusions, estimating functions,

\hspace{2.2cm}eigenfunctions.

\section{Introduction}\label{Sec:intro}
Diffusion processes have been widely used in applications, and statistical inference for them have been extensively studied. We refer readers to, for example, \cite{kutoyants2004statistical}, \cite{HSorensen04Survey}, \cite{AHS}, \cite{BibbyJacobsenSorensen} and \cite{KesslerLindnerSorensen12} for survey of various techniques in the literature. However, there are also many applications, especially in finance and economics, in which diffusion models are not adequate to describe the data due to the presence of jumps. See for example, \cite{ait2009testing}, \cite{ait2009estimating}, \cite{ait2011testing}, \cite{ait2012testing}, \cite{todorov2010activity}, \cite{TodorovTauchen} for various non-parametric methods to test existence of jumps and evidence for jumps in important applications.

A useful way to construct Markovian jump processes is to apply Bochner's subordination to diffusion processes. This classical technique, originally introduced in \cite{Bochner} in the semigroup context, corresponds to a stochastic time change using an independent nonnegative L\'{e}vy process (a.k.a., L\'{e}vy subordinator) as the random clock (a detailed account of Bochner's subordination can be found in \cite{schilling2012bernstein}, Chapter 13). Let $X$ be a time-homogeneous diffusion, and $T$ be a L\'{e}vy subordinator, independent of $X$. The time changed process $(X_{T_t})_{t\ge0}$ is called subordinate diffusion, and it is Markovian and time-homogeneous due to the the independent increment and stationary increment property of the L\'{e}vy subordinator, respectively. Since $T$ generally jumps, jumps are created in $X_T$. Depending on whether $T$ has drift or not, $X_T$ is a jump-diffusion or a pure jump process. Jumps of $X_T$ are in general state-dependent and could exhibit a variety of interesting behavior, making the time-changed process an appropriate model in many applications. For example, if $X$ is a mean-reverting diffusion, then jumps of $X_T$ are mean-reverting as well (see \cite{LiLinetsky}), and if $X$ moves in a finite interval, then $X_T$ does not jump outside the same interval. In addition, jumps of $X_T$ could have finite or infinite activity and finite or infinite variation. Some recent high-frequency non-parametric statistical analysis shows that some financial variables follow a pure jump process with infinite jump activity and infinite jump variation (see e.g., \cite{TodorovTauchen}) and subordinate diffusions provide natural parametric candidates for modeling them. Successful applications of subordinate diffusions have already been found in finance. See for example, \cite{MendozaCarrLinetsky,LiLinetsky,LLZVIX} and the discussions in \cite{LiLiMendozaAddSub,LiLinetskyOS,LiLinetskyFP}. But these references focus on option pricing. In the special case where $X$ is a Brownian motion, $X_T$ is a L\'{e}vy process (\cite{ContTankov}, Section 4.4) and hence one can view subordinate diffusions as a natural generalization of many L\'{e}vy processes by time changing more general diffusions.

This paper considers parametric inference for discretely observed ergodic subordinate diffusions, which has not been studied in the literature. Here only the value of the time-changed process $X_T$ is observed at a discrete set of times, and both the diffusion $X$ and the L\'{e}vy subordinator $T$ cannot be observed. This setting fits the applications we have in mind. For example, in \cite{LiLinetsky}, the subordinate Ornstein-Uhlenbeck process is used to model the commodity spot price, which is the only quantity that can be observed in practice. Our aim is to develop an estimation method that is both computationally and statistically efficient, and it is applicable for a general class of $X$ and $T$. We develop a two-step estimation procedure based on estimating functions that meets all these requirements. As an application, we show that a subordinate diffusion provides good fit to the historical data of VIX (CBOE Volatility Index). VIX is commonly regarded as the market's fear gauge and there are many actively traded derivatives written on VIX, which are very important investment and hedging tools (see the review in \cite{LLZVIX}). To fit the VIX data, diffusion models are proposed in \cite{goard2013stochastic}. We show that our subordinate diffusion model which contains jumps performs significantly better. In the rest of the introduction, we provide key background information for subordinate diffusions and discuss issues involved and related literature.

\subsection{Subordinate Diffusions}
Consider a diffusion process $X$ living on an interval $I$ with end-point $l$ and $r$ ($-\infty\le l<r\le \infty$). We denote its drift and diffusion coefficient by $\mu(x)$ and $\sigma(x)$, respectively, and its infinitesimal generator by $\mathcal{G}$, which is an operator defined on a dense subset of $L^2(I,q):=\{f\ \text{measurable}: \int_l^r f^2(x)q(x)dx<\infty\}$, where $q(x)$ is the stationary density of $X$ (see \eqref{eq:DiffStationary}). The following assumption on $X$ is made in this paper.
\begin{assumption}\label{assump:diff}
(1) $X$ is ergodic with stationary density
\begin{equation}\label{eq:DiffStationary}
q(x) = \frac{{m(x)}}{{\int_l^r {m(y)dy} }},\
m(x) = \frac{\exp \left[ {\int_{}^x { \frac{{2\mu (y)}}{{{\sigma^2}(y)}}} dy} \right]}{{{\sigma ^2}(x)}}.
\end{equation}
Here $m(x)$ is the speed density of $X$.

\noindent(2) The spectrum of $\mathcal{G}$ is purely discrete.

\noindent(3) $\mu(x)$ and $\sigma(x)$ are twice continuously differentiable and $\sigma(x)>0$ on any interval $[a,b] \subset I$.
\end{assumption}

Sufficient conditions for a diffusion to satisfy Assumption \ref{assump:diff} (1) can be found in \cite{kessler1999estimating}, Condition 4.1.
The purely discrete spectrum assumption holds for many ergodic diffusions used in applications, with well-known examples including the Ornstein-Uhlenbeck process, Feller's square-root process and the Jacobi process. \cite{linetsky2007spectral}, Theorem 3.2 provide conditions that imply purely discrete spectrum for diffusions (see also \cite{hansen1998spectral}). Assumption \ref{assump:diff} (3) is a technical condition that is needed in solving the identification problem for subordinate diffusions. Under Assumption \ref{assump:diff}, if $l\in I$, then $l$ must be a reflecting boundary, otherwise it is inaccessible. The same conclusion holds for $r$.

Next consider a L\'{e}vy subordinator $T$, which is a nonnegative L\'{e}vy process and assume that it is independent of $X$. Its Laplace transform is given by the well-known L\'{e}vy-Khintchine formula (e.g., \cite{ContTankov}, Eq.(4.5))
\begin{equation}\label{eq:LTLevySub}
E[e^{-\lambda T_t}]=e^{-\phi(\lambda)t},\
\phi(\lambda)=\gamma\lambda+\int_{(0,\infty)}(1-e^{-\lambda
s})\nu(ds),\ \lambda\ge 0,
\end{equation}
where $\gamma\geq0$ is the drift of $T$ and $\nu$ is called the L\'{e}vy
measure with $\int_{(0,\infty)}(s\wedge1)\nu(ds)<\infty$. The function $\phi(\lambda)$ is known as the Laplace exponent in the literature. A commonly used class of L\'{e}vy subordinators is the tempered stable family, in which
\begin{equation}\label{eq:TSLevyMeasure}
\nu(ds)=Cs^{-p-1}e^{-\eta s},\ C>0, 0<p<1, \eta>0.
\end{equation}
When $p=\frac{1}{2}$, the subordinator is the inverse Gaussian process (\cite{BarndorffNielsen}), a popular choice in finance. For tempered stable subordinators,
\begin{equation}\label{eq:TSLaplaceExponent}
\phi(\lambda)=\gamma\lambda - C\Gamma(-p)[(\lambda+\eta)^p-\lambda^p],
\end{equation}
where $\Gamma(\cdot)$ is the gamma function.

A subordinate diffusion $Y$ is defined as $Y_t=X_{T_t}$. In general, $Y$ is a jump-diffusion or a pure jump process depending on whether $\gamma$ is positive or zero. Denote the infinitesimal generator of the subordinate diffusion by $\mathcal{G}^\phi$. Then using the Phillips theorem (\cite{schilling2012bernstein}, Theorem 13.6), one can show that $\gD(\gG^\phi)\subseteq\gD(\gG)$, and for $f \in C_c^2(I)$ (a fully rigorous proof is given in \cite{LiLiMendozaAddSub}, Theorem 4.2),
\begin{align}
{\rm{\gG}}^\phi f(x)
&= \frac{1}{2}{({\sigma ^\phi }(x))^2}f''(x) + {\mu ^\phi }(x)f'(x) \notag \\
&+ \int_{z \ne 0} {\left( {f(x + z) - f(x) - {1_{\{ |z| \le 1\} }}zf'(x)} \right)} {\Pi ^\phi }(x,dz)
\end{align}
where for $t \ge 0$ and $x \in (l,r)$,
\begin{align}
&{\sigma ^\phi }(x) = \sqrt {\gamma } \sigma (x),\\
&{\mu ^\phi }(x) = \gamma \mu (x) + \int_{(0,\infty )} {\left( {\int_{\{ |z|  \le 1\} } {zp(\tau ,x,x + z)dz} } \right)} \nu(d\tau),\\
&{\Pi ^\phi }(x,dz) = {\pi ^\phi }(x,z)dz,\qquad
{\pi ^\phi }(x,z) = \int_{(0,\infty )} {p(\tau ,x,x + z)} \nu (d\tau ).\label{eq:subdiff-jump-measure}
\end{align}
Here $p(t,x,y)$ is the transition density of diffusion $X$ and we extend the definition of $p(t,x,y)$ to $y\notin I$ by defining $p(t,x,y)=0$.
It can be proved that ${\Pi ^\phi}(x,dz) $ is a L{\'e}vy-type measure, i.e., $\int_{z\ne 0}(z^2\wedge 1)\Pi^\phi(x,dz)<\infty$. The jump intensity $\pi^\phi(x,z)$ is clearly state dependent in general and could exhibit a variety of interesting behavior, making subordinate diffusions good candidates for jump modeling.

Let $\lambda_n\le 0$ be the $n$-th eigenvalue of the diffusion generator $\mathcal{G}$ and $\varphi_n(x)$ is the associated eigenfunction, i.e., $\mathcal{G}\varphi_n(x)=\lambda_n\varphi_n(x)$. Under Assumption \ref{assump:diff}, purely discreteness of the spectrum also implies that all eigenvalues are simple (\cite{linetsky2007spectral}, Theorem 3.2), so we have $0\ge \lambda_0>\lambda_1>\lambda_2>\cdots$. Let $\mathcal{P}_t$ be the transition operator of $X$, i.e., $\mathcal{P}_tf(x)=E_x[f(X_t)]$ and $\mathcal{P}_t$ is defined on $L^2(I,q)$. Then the spectrum of $\mathcal{P}_t$ is also discrete and $\mathcal{P}_t\varphi_n(x)=e^{\lambda_nt}\varphi_n(x)$ ($\mathcal{P}_t$ and $\mathcal{G}$ share the same set of eigenfunctions). We normalize $\varphi_n(x)$ such that $\int_l^r \varphi^2_n(x)q(x)dx=1$, where $q(x)$ is the diffusion stationary density defined in \eqref{eq:DiffStationary}. We also have $\int_l^r \varphi_n(x)\varphi_m(x)q(x)dx=0$ for $n\ne m$, that is, different eigenfunctions are orthogonal w.r.t. $q(x)$. The set of eigenfunctions $\{\varphi_n(x):n=0,1,2,\cdots\}$ forms an orthonormal basis of $L^2(I,q)$.

A key observation that will be used in developing the estimation method is that for $Y$, the spectrum of its generator $\mathcal{G}^\phi$ (defined on a dense subset of $L^2(I,q)$) and its transition operator $\mathcal{P}^\phi_t$ (defined on $L^2(I,q)$) are also purely discrete, and
\begin{equation}\label{eq:SubDiffGenEigen}
\mathcal{G}^\phi\varphi_n(x)=-\phi(-\lambda_n)\varphi_n(x),\ \mathcal{P}^\phi_t\varphi_n(x)=e^{-\phi(-\lambda_n)t}\varphi_n(x).
\end{equation}
This equation shows that subordination preserves the set of eigenfunctions and only changes the eigenvalues using the Laplace exponent of the subordinator. The proof of this fact can be found in \cite{linetsky2007spectral}, p.283. Let $p^\phi(t,x,y)$ be the transition density of $Y$. \cite{LiLinetskyFP}, Proposition 2.4 shows that, under the condition $\sum_{n=0\infty}e^{\lambda_nt}<\infty$, if either $\gamma>0$ or when $\gamma=0$, $\varphi_n(x)$ is bounded on any compact set of $x$ for all $n$ and $\sum_{n=0}^\infty e^{-\phi(-\lambda_n)t}<\infty$, then $p^\phi(t,x,y)$ admits the following bilinear eigenfunction expansion which converges uniformly on compacts for $x$ and $y$:
\begin{equation}\label{eq:SubDiffTPD}
p^\phi(t,x,y)= q(y)\sum_{n = 0}^\infty e^{-\phi(-\lambda_n)t}\varphi_n(x)\varphi_n(y).
\end{equation}
Lemma \ref{lemma:ergodicity} in Section 2 shows that under Assumption \ref{assump:diff}, $Y$ is ergodic and $p^\phi(t,x,y)$ converges to $q(y)$ as $t\to\infty$. As for $n\ge 1$, $\lambda_n<0$, we have $e^{-\phi(-\lambda_n)t}\varphi_n(x)\varphi_n(y)\to 0$ as $t\to\infty$. Subsequently, we must also have
\begin{equation}\label{eq:0-th-eigen}
\varphi_0 (x)\equiv 1,\ \lambda_0=0,
\end{equation}
for $p^\phi(t,x,y)$ to converge to $q(y)$.

\subsection{Issues and Related Literature}
The first issue we need to address is identifiability of subordinate diffusions. In general, one cannot hope to identify $X$ and $T$ uniquely given only the data of $Y$ (an example is given in Section \ref{sec:identification}). Using spectral theory, we show that the characteristics of the diffusion and the subordinator can be identified up to scale. This implies that to estimate the parameters of $Y$, the scale needs to be fixed first, but the law of $Y$ does not change when the scale varies.

The transition density $p^\phi(t,x,y)$ is given by \eqref{eq:SubDiffTPD}. In general, $\lambda_n$ and $\varphi_n(x)$ are unknown so the method of maximum likelihood estimation cannot be applied. Even when explicit expressions are available for them, computing the expansion for $p^\phi(t,x,y)$ can be demanding especially when the time step between two observations is small. For example, when $Y$ is the subordinate Ornstein-Uhlenbeck process (\cite{LiLinetsky}), $\lambda_n$ and $\varphi_n(x)$ are known, but in our numerical experiment it could take several thousand or even over 10,000 terms for the partial sum in \eqref{eq:SubDiffTPD} to converge to an acceptable level of accuracy when $t$ is one day. In financial applications, typically daily or even higher frequency data is used.

We propose an estimation method for subordinate diffusions based on estimating functions. An overview of the estimating function approach for diffusions can be found in \cite{SorensenReview}, \cite{BJSEstFun}. It is shown that this is a statistically efficient method for diffusions if the estimating functions are appropriately chosen. In particular, when analytical expressions for $\lambda_n$ and $\varphi_n(x)$ are available,  \cite{kessler1999estimating} (hereafter KS) propose to construct martingale estimating functions based on the eigenfunctions and they work well in applications (see \cite{larsen2007diffusion} for the application of this method to estimate the dynamics of exchange rates in a target zone). For subordinate diffusions, if we have analytical formulas for $\lambda_n$ and $\varphi_n(x)$, we can directly apply the KS idea to construct martingale estimating functions based on \eqref{eq:SubDiffGenEigen}, but compared to estimating $X$, $Y$ has more parameters so computation takes longer time. When $\lambda_n$ and $\varphi_n(x)$ do not have analytical expressions, since there exist numerical algorithms for the Sturm-Liouville problem that can achieve high level of accuracy, we can compute $\lambda_n$ and $\varphi_n(x)$ numerically. If the KS approach were followed, numerical computation of $\lambda_n$ and $\varphi_n(x)$ is needed in every iteration (note that the estimator is found by solving an equation through iterations), which is time-consuming.

This observation leads us to propose a two-step procedure that is computationally much more efficient. In Step 1, we use estimating functions proposed in \cite{conley1997short} to estimate diffusion parameters (up to scale). \cite{conley1997short} considers how to estimate the parameters of a diffusion under random sampling (but they do not estimate the parameters associated with the random sampling scheme). They propose estimating functions that only involve diffusion parameters using the randomly sampled data. Since deterministically sampled data of a subordinate diffusion can be viewed as randomly sampled data of the background diffusion, we can apply their estimating functions, and the estimator can be computed fast. In Step 2, we estimate the subordinator parameters using the eigenfunction-based estimating function $(1/N)\sum_{n=1}^N\sum_{i=1}^Mw_i(y_{n-1},\theta_S)(\varphi_i(y_n)-e^{-\phi(-\lambda_i,\theta_S)}\varphi_i(y_{n-1}))$, where $\{y_0,y_1,\cdots,y_N\}$ is the data for $Y$, $\theta_S$ is the vector of subordinator parameters and $w_i$ is a column vector with the same length as $\theta_S$. Since the diffusion parameters have been estimated in Step 1, all $\lambda_i$ and $\varphi_i$ are determined and do not change in iterations to find the estimator  $\hat{\theta}_S$. How to choose $W:=(w_1,\cdots,w_M)$ is important for the method's statistical efficiency. We can set $W$ according to the ``optimal weight'' formula in \cite{kessler1999estimating} (hereafter KS weight). If the diffusion parameters estimated in Step 1 are the true values, the KS weight is optimal in the sense that the covariance matrix for $\hat{\theta}_S$ is minimized. Since diffusion parameters are estimated and thus contain errors, the KS weight is not optimal. In this paper we obtain the formula for the optimal weight, which is nevertheless difficult to compute. In our simulation study, we numerically compare the standard error for $\theta_S$ under the optimal weight and the KS weight. We find that the results are very close. Given the ease of calculating the KS weight, we use it in our method instead of the optimal one. The idea of combining different types of estimating functions is also used in \cite{bibby2001simplified} for estimating a discretely observed diffusion with a high-dimensional parameter. There, simple estimating functions proposed by \cite{kessler2000simple} and martingale estimating functions developed in \cite{bibby1995martingale} are combined to simplify the estimation procedure.

Methodologically, our paper improves \cite{kessler1999estimating} and \cite{bibby2001simplified} in the following aspects.
\begin{itemize}
\item[(1)] \cite{kessler1999estimating} only considers the situation where eigenvalues and eigenfunctions are analytically known. We deal with the general case with unknown eigenvalues and eigenfunctions, and show how to calculate the eigenfunction-based martingale estimating function numerically in an efficient way. We also develop consistency and asymptotic normality results considering the effect of numerical approximation. These results are not directly implied by the existing asymptotic theory for estimating functions which assumes that they can be computed exactly.

\item[(2)] \cite{bibby2001simplified} did not address the issue of obtaining the optimal weighting matrix for the martingale estimating function when it is combined with other estimating functions. We solved this problem in our context.
\end{itemize}

The present work is related to a growing literature on estimating time-changed L\'{e}vy processes, which are also constructed by time change and are very popular in modeling asset prices (see \cite{carr2004time}). A time-changed L\'{e}vy process is constructed as $L_{T_t}$, where $L$ is a L\'{e}vy process with L\'{e}vy measure $\ell$ and $T_t=\int_{0}^t A_sds$. Such time change is absolutely continuous, and the process $A$ is often called the activity rate process, which is used to introduce stochastic volatility into the L\'{e}vy model. We do not attempt to survey the rather extensive literature on the estimation of time-changed L\'{e}vy processes, but instead mention a few works. See, e.g., \cite{figueroa2009nonparametric,figueroa2011central}, \cite{belomestny2011statistical} for non-parametric estimation of $\ell$, and \cite{bull2014estimating} for inference of $A_s$, among other works. Subordinate diffusions and time-changed L\'{e}vy processes are constructed using different background processes (diffusions for the former and L\'{e}vy processes for the latter) and different time changes (L\'{e}vy subordinators are used as the time change for the former, which are not absolutely continuous). Thus, subordinate diffusions generally do not belong to the class of time-changed L\'{e}vy processes. In addition, subordinate diffusions exhibit richer jump behavior than time-changed L\'{e}vy processes because they could have state-dependent jumps (the compensator of the random jump measure of $Y$ is $\Pi^\phi(Y_{t-},dz)dt$; see the expression for $\Pi^\phi(x,dz)$ in \eqref{eq:subdiff-jump-measure}), while jumps in time-changed L\'{e}vy processes are state-independent (the compensator of the random jump measure of $L_T$ is $A_t\ell(dz)dt$, which is independent of $L_{T_{t-}}$). Unlike time-changed L\'{e}vy processes, which have stochastic volatility and can generate the volatility clustering phenomenon, subordinate diffusions do not possess such a feature. To incorporate it, one can further time change a subordinate diffusion by an absolutely continuous process in the form of $\int_{0}^t A_sds$ (see \cite{LiLinetsky}). How to estimate these time-changed subordinate diffusions is an interesting problem for future research.

\subsection{Organization of the Paper}
The rest of the paper is organized as follows. Section \ref{sec:identification} solves the identification problem for subordinate diffusions. In Section \ref{sec:Estimation}, we present the two-step procedure to estimate subordinate diffusions by first assuming that $\lambda_n$ and $\varphi_n(x)$ are known and derive the optimal weight for the eigenfunction-based martingale estimating function. We then consider the general situation in which $\lambda_n$ and $\varphi_n(x)$ are unknown and we show how to obtain the estimator using the numerical approximation of $\lambda_n$ and $\varphi_n(x)$. In Section \ref{sec:CAN}, we develop asymptotic analysis of our estimators considering the effect of numerical approximations. Under regularity conditions, we show that our estimator is consistent and asymptotically normal. Section \ref{sec:num} contains various numerical examples and an application to VIX data. Proofs are collected in the appendix except those in Section \ref{sec:CAN}. Section \ref{sec:conclusions} provides a summary and discusses future research.

To conclude the introduction, we fix some notations. $'$ denotes the transpose of a vector or a matrix. For a $M\times 1$ vector function $F$ and a $p\times 1$ parameter $\theta$, $\partial_\theta F$ is a $M\times p$ matrix with element $(\partial_\theta F)_{i,j}=\partial_{\theta_j}F_i $. In particular, when $F$ is a scalar function, $\partial_\theta F$ is a row vector.

\section{Identifiability of Subordinate Diffusions}\label{sec:identification}
Since we are only given the data of $Y$, it is expected that $X$ and $T$ cannot be uniquely identified. Below we present an example.
\begin{example}\label{eg:IG-SubOU}
Let $X$ be an Ornstein-Uhlenbeck (OU) diffusion, i.e.,
\begin{equation}\label{eq:OU}
dX_t=\kappa(\vartheta-X_t)dt+\sigma dW_t,
\end{equation}
with $\kappa,\sigma>0$. Its stationary density is given by
\begin{equation}\label{eq:OUStationary}
q(x) = \sqrt {\frac{\kappa }{{\pi {\sigma ^2}}}} {e^{ - \kappa \frac{{{{(x - \vartheta )}^2}}}{{{\sigma ^2}}}}}.
\end{equation}
It is well known that for the OU process (\cite{karlin1981second}),
\begin{equation}\label{eq:OU-eigen}
\lambda_n = -\kappa n,\ \varphi_n(x)= \frac{1}{{\sqrt {{2^n}n!} }}H_n\left(\frac{\sqrt{\kappa}}{\sigma}(x-\vartheta)\right),\ n=0,1,\cdots,
\end{equation}
where $H_n(x)$ is the Hermite polynomial of order $n$, and $\varphi_n(x)$ satisfies $\int_{\mathbb{R}}\varphi_n^2(x)q(x)dx=1$.

Let $T$ be an inverse Gaussian subordinator with drift $\gamma$. Its L\'{e}vy measure is given by $\nu(ds)=Cs^{-\frac{3}{2}}e^{-\eta s}ds$ with $C>0$, $\eta>0$, and $\phi(\lambda)=\gamma\lambda - C\Gamma(-\frac{1}{2})[\sqrt{\lambda+\eta}-\sqrt{\lambda}]$. We call $Y$ an IG-SubOU process for short.

In this case one can verify the conditions in Proposition 2.4 of \cite{LiLinetskyFP}, so we have the bilinear eigenfunction expansion \eqref{eq:SubDiffTPD} for $p^\phi(t,x,y)$. Using the explicit expressions for $\lambda_n$, $\varphi_n(x)$, $\phi(\lambda)$, and the expansion, it is easy to verify that for any $c>0$, $(\vartheta,c\kappa,\sqrt{c}\sigma,\gamma/c,C/\sqrt{c},c\eta)$ gives the same $p^\phi(t,x,y)$.
\end{example}

Let $\theta (x)=(\mu (x),\sigma^2 (x))$. For the L\'{e}vy measure $\nu$ of $T$, define
\begin{equation}\label{eq:omega}
\omega(s):=\nu(s,\infty),\quad \hat{\omega}(\lambda)=\int_{0}^{\infty}e^{-\lambda s}\omega(s)ds,\ \lambda>0.
\end{equation}
We call $(\theta(x),\gamma,\hat{\omega}(\lambda))$ the characteristics triplet of $Y$. Example \ref{eg:IG-SubOU} already shows that in the case of IG-SubOU process, given a characteristics triplet, appropriate scaling does not alter the law of the process. This observation holds more generally provided that we have the bilinear eigenfunction expansion for $p^\phi(t,x,y)$. Furthermore, using spectral theory, we show that given two characteristics triplets, that they yield the same transition probability density implies that they are related by appropriate scaling. We need the following lemma on the ergodicity of subordinate diffusions, which will also be used in proving consistency of our estimator.

\begin{lemma}\label{lemma:ergodicity}
The continuous time process $\{Y_t,t\ge 0\}$ and the sampled process $\{Y_{t_i}:t_i=i\Delta ,i=0,1,\cdots\}$ are ergodic under Assumption \ref{assump:diff}. The stationary density of $Y$ is given by $q(x)$ defined in \eqref{eq:DiffStationary}.
\end{lemma}

\begin{thm} \label{thm:idenPPS}
Consider two characteristics triplet $(\theta_i(x),\gamma_i,\hat{\omega}_i(\lambda))$ ($i=1,2$) and denote the corresponding transition probability density by $p^\phi_i(t,x,y)$. Under Assumption \ref{assump:diff}, $p^\phi_1(t,x,y)$ and $p^\phi_2(t,x,y)$ are identical implies that there exists a constant $c > 0$ such that
\begin{equation}\label{eq:idenPPSConditions}
\theta_1(x)=\frac{1}{c}\theta_2(x),\ \gamma_1=c\gamma_2,\ \hat{\omega}_1(-\lambda_{n})=c\hat{\omega}_2(-c\lambda_{n})\ \text{for all}\ n,
\end{equation}
where $\lambda_{n} $ is the $n$-th eigenvalue of the generator of the diffusion with characteristic $\theta_1(x)$. Now, suppose that \eqref{eq:idenPPSConditions} holds for some constant $c>0$. Under Assumption \ref{assump:diff} and the condition $\sum_{n=0}^\infty e^{\lambda_nt}<\infty$, if either $\gamma>0$ or when $\gamma=0$, $\varphi_n(x)$ is bounded on any compact set of $x$ for all $n$ and $\sum_{n=0}^\infty e^{-\phi(-\lambda_n)t}<\infty$,
then $p^\phi_1(t,x,y)$ and $p^\phi_2(t,x,y)$ are identical.
\end{thm}

For a given subordinator, using the explicit form of its L\'{e}vy measure, we can simplify the condition $\hat{\omega}_1(-\lambda_{n})=c\hat{\omega}_2(-c\lambda_{n})$ to obtain explicit conditions on the parameters of the subordinator. Below we consider the important class of tempered stable subordinators.
\begin{corollary}\label{cor:TSS}
Suppose the L\'{e}vy measure $\nu(d\tau)$ belongs to the tempered stable family (see \eqref{eq:TSLevyMeasure}). \eqref{eq:idenPPSConditions} is equivalent to $\theta_1(x)=\frac{1}{c}\theta_2(x)$, $\gamma_1=c\gamma_2$, $p_1=p_2$, $\eta_1=\frac{1}{c}\eta_2$ and $C_1=c^{p_1}C_2$.
\end{corollary}

The result in Example \ref{eg:IG-SubOU} for the IG-SubOU process becomes a special case of Corollary \ref{cor:TSS} with $p_1=p_2=\frac{1}{2}$. Theorem \ref{thm:idenPPS} implies that, to estimate the parameters of a subordinate diffusion, a scale needs to be fixed first. In Example \ref{eg:IG-SubOU}, we can, for example, fix $\sigma=1$ and estimate the remaining parameters of the IG-SubOU process.

\section{A Two-Step Estimation Procedure for Subordinate Diffusions}\label{sec:Estimation}
Let $\theta_1$ be a $p_1\times 1$ vector for the parameters of the diffusion $X$ and $\theta_2$ be a $p_2\times 1$ vector for the parameters of the subordinator $T$. Put $\theta=(\theta_1',\theta_2')'$, which is a $(p_1+p_2)\times 1$ vector for the parameters of $Y$. The data of $Y$ is given by $\{y_{t_i}:i=0,1,\cdots,n\}$ with $t_i=i\Delta$. In our estimation procedure, we use two estimating functions
%two types of moment conditions
%\begin{equation}\label{eq:MomentCond}
%E_{\theta}[f_1(Y_{t_0},Y_{t_1};\theta_1)]=0,\ E_{\theta}[f_2(Y_{t_0},Y_{t_1};\theta_1,\theta_2)]=0,
%\end{equation}
\begin{equation} \label{eq: F1}
F_{n,1}(\theta_1)=\sum_{i=1}^{n}f_1(y_{t_{i-1}},y_{t_i};\theta_1),
\end{equation}
where $f_1$ is a $p_1\times 1$ vector function, and
\begin{equation} \label{eq: F2}
F_{n,2}(\theta_1,\theta_2)=\sum_{i=1}^{n}f_2(y_{t_{i-1}},y_{t_i};\theta_1,\theta_2),
\end{equation}
where $f_2$ is a $p_2\times 1$ vector function. We choose $f_1$ based on moment conditions developed in \cite{conley1997short} and $F_{n,2}$ is a martingale estimating function based on eigenfunctions. The estimation consists of two steps. In the first step, $\hat{\theta}_{n,1}$, the estimator of $\theta_1$, is obtained by solving $F_{n,1}(\theta_1)=0$. Then, in the second step, we find $\hat{\theta}_{n,2}$, the estimator of $\theta_2$, by solving $F_{n,2}(\hat{\theta}_{n,1},\theta_2)=0$. To simplify the notation, we will also write $F_{n,2}(\theta_1,\theta_2)$ as $F_{n,2}(\theta)$ below.

\subsection{The Choice of Moment Conditions} \label{momentSec}
\cite{conley1997short} estimates the parameters of a diffusion under random sampling. They assume that the random sampling process is increasing and independent of the underlying diffusion and has stationary increments. Under this assumption, two types of moment conditions are proposed. A deterministic sample of the subordinate diffusion $Y$ can be viewed as a random sample of the diffusion $X$. Furthermore, in our set-up, the random sampling process $T$ clearly satisfies the assumption in \cite{conley1997short}. Hence we can adopt the following two types of moment conditions proposed there ($\gQ$ is the stationary distribution of $Y$ with density $q$ and $E_{\gQ}$ denotes taking expectation with initial distribution equal to $\gQ$)
\begin{equation} \label{eq: diffCon1}
E_{\gQ}[\mu(Y_t)g'(Y_t)+\frac{1}{2} \sigma^2(Y_t)g''(Y_t)]=0 \quad \text{for any}\ g\in \gD(\gG),
\end{equation}
and
\begin{equation} \label{eq: diffCon2}
E_{\gQ}[\gA \Psi (Y_{t+\Delta},Y_t)-\gA' \Psi(Y_{t+\Delta},Y_t)]=0,
\end{equation}
where
\begin{equation}
\gA \Psi(x,y)=\mu(x) \partial_x\Psi(x,y)+\frac{1}{2}\sigma^2(x) \partial_x^2 \Psi(x,y),\ \gA' \Psi(x,y)=\mu(y) \partial_y\Psi(x,y)+\frac{1}{2}\sigma^2(y) \partial_y^2 \Psi(x,y).
\end{equation}
The function $\Psi (x,y)$ satisfies that (1) for each $x\in (l,r)$, $\Psi(x,\cdot)\in D$ is bounded and continuous and for each $y\in (l,r)$, $\Psi(\cdot,y)\in D$ is bounded and continuous; (2) $\gA \Psi (\cdot,y)$ is bounded and continuous for all $y \in (l,r)$ and $\gA' \Psi (x,\cdot)$ is bounded and continuous for all $x \in (l,r)$. An efficient choice of test function $g$ in \eqref{eq: diffCon1} is given by (see \cite{hansen1995back}, \cite{conley1997short}, \cite{kessler2000simple})
\begin{equation}\label{eq:score-vec}
g(y;{\theta _1}) = \frac{\partial }{{\partial {\theta _1}}}\log q(y;{\theta _1}).
\end{equation}
To construct the vector function $f_1$ in the estimating function \eqref{eq: F1}, we select its components from the moment conditions \eqref{eq: diffCon1} and \eqref{eq: diffCon2}.

We use $(\lambda_m,\varphi_m(x))$ ($1\le m\le M$) to construct $f_2$. Recall that for each $m$, $\gP_\Delta^\phi\varphi_m(x)=e^{-\phi(-\lambda_m)\Delta}\varphi_m(x)$,
and using the tower law, $E[\varphi_m(Y_{t+\Delta})-e^{-\phi(-\lambda_m)\Delta}\varphi_m(Y_t)]=0$, which holds for any initial distribution for process $Y$. We can combine these moment conditions together.
Let $c_m$ be a $p_2\times 1$ vector function ($1\le m\le M$), then
\begin{equation}
E\left[\sum_{m= 1}^M {c_m}({Y_t})\left(\varphi_m(Y_{t + \Delta}) - e^{-\phi(-\lambda_m)t}{\varphi_m}(Y_t)\right)\right] = 0.
\end{equation}
Note that we do not use $(\lambda_0,\varphi_0(x))$ in the moment condition due to \eqref{eq:0-th-eigen}. Let $V(y_1,y_2;\theta)$ be a $M \times 1$ vector function with each element
\begin{equation}
V_m(y_1,y_2;\theta) = \varphi_m(y_2;\theta) - e^{-\phi(-\lambda_m;\theta)\Delta}\varphi_m(y_1;\theta),
\end{equation}
and $W(y_1;\theta)$ is a $M\times p_2$ matrix. We put $f_2(y_1,y_2;\theta)=W'(y_1;\theta)V(y_1,y_2;\theta)$ (recall that $W'$ is the transpose of $W$). Then,
\begin{equation} \label{eq: F2-1}
F_{n,2}(\theta)=\sum_{i=1}^{n}W'(y_{t_{i-1}};\theta)V(y_{t_{i-1}},y_{t_i};\theta).
\end{equation}
It is also easy to see that $F_{n,2}$ is a martingale. Such martingale estimating function based on eigenfunctions is first proposed by \cite{kessler1999estimating} to estimate a discretely sampled diffusion with analytical expressions for the eigenvalues and the eigenfunctions. The choice of the weighting matrix $W$ is crucial for this method's efficiency, which we discuss next.

\subsection{The Choice of the Weighting Matrix}
Let's first assume that the diffusion parameter $\theta_1$ is known. Then, to determine the optimal weighting matrix in \eqref{eq: F2-1} in the sense of \cite{godambe2010quasi}, we can follow \cite{kessler1999estimating}. Adapting Eq.(3.3) in \cite{kessler1999estimating} to our setting, we obtain a weighting matrix which solves the following linear system (we denote the solution by $W_{KS}$ and refer to it as the KS weight hereafter)
\begin{equation} \label{eq: solveW}
P(y;\theta)W_{KS}(y;\theta)=Q(y;\theta),
\end{equation}
where $P$ is a $M \times M$ matrix and $Q$ is a $M \times p_2$ matrix with
\begin{align}
{P_{i,j}}(y;\theta) &= \int_l^r {{\varphi _i}(x;\theta_1){\varphi _j}(x;\theta_1){p^\phi }(t,y,x;\theta )dx}\\
&- {e^{ - \phi ( - {\lambda _i};\theta )t}}{e^{ - \phi ( - {\lambda_j};\theta )t}}{\varphi _i}(y;\theta_1){\varphi _j}(y;\theta_1),\label{eq:P-mat}\\
{Q_{i,j}}(y;\theta) &= \frac{\partial }{{\partial {\theta_{2,j}}}}{e^{ - \phi ( - {\lambda _i};\theta )t}}{\varphi _i}(y;\theta_1).\label{eq:Q-mat}
\end{align}
Since the diffusion parameters also need to be estimated, $W_{KS}$ computed by \eqref{eq: solveW} is not the true optimal weighting matrix.

Now we derive the optimal weighting matrix. We will make precise the meaning of being ``optimal'' below. First, we define one vector and two matrices. Let
\begin{align}
&{F_n} = \left[ {\begin{array}{*{20}{c}}
	{{F_{n,1} }(\theta_1)}\\
	{{F_{n,2} }(\theta)}
	\end{array}} \right],\\
&S_n: = \left[ {\begin{array}{*{20}{c}}
	S_{1,1}&S_{1,2}\\
	S_{2,1}&S_{2,2}
	\end{array}} \right]=E[{F_n}{F'_n}] = \left[ {\begin{array}{*{20}{c}}
	{E[{F_{n,1} }(\theta_1)F{'_{n,1} }(\theta_1)]}&{E[{F_{n,1} }(\theta_1)F{'_{n,2} }(\theta)]}\\
	{E[{F_{n,2} }(\theta)F{'_{n,1} }(\theta_1)]}&{E[{F_{n,2} }(\theta)F{'_{n,2} }(\theta)]}
	\end{array}} \right],\\
&D_n:= \left[ {\begin{array}{*{20}{c}}
	D_{1,1}&0\\
	D_{2,1}&D_{2,2}
	\end{array}} \right]=E[{{\dot F}_n}]  = \left[ {\begin{array}{*{20}{c}}
	{E[\partial _{\theta_1}{F_{n,1} }(\theta_1)]}&0\\
	{E[\partial_{\theta_1} {F_ {n,2}}(\theta)]}&{E[\partial_{\theta_2} {F_{n,2} (\theta)}]}
	\end{array}} \right].
\end{align}
To simplify the notation, we suppress the dependence on $\theta$ and $n$ in $S_{i,j}$ and $D_{i,j}$. Assuming that $D_{1,1}$ and $D_{2,2}$ are invertible, $D_n^{-1}$ can be represented as
\begin{equation}
\left[ {\begin{array}{*{20}{c}}
	{{D_{1,1}^{ - 1}}}&0\\
	{ - {D_{2,2}^{ - 1}}D_{2,1}{D_{1,1}^{ - 1}}}&{{D_{2,2}^{ - 1}}}
	\end{array}} \right] .
\end{equation}
In Section \ref{sec:CAN}, under certain regularity conditions, we will prove that
\begin{equation}
\sqrt{n}(\hat{\theta}_n-\bar{\theta}) \to \mathcal{N}(0,\Sigma),
\end{equation}
where $\hat{\theta}_n$ is the estimator for $\theta$ and $\bar{\theta}$ is its true value, and
\begin{align}
\Sigma&=\lim_{n\to\infty} nD_n^{-1}S_n (D_n')^{-1} \\
&=\lim_{n\to\infty} n\left[ {\begin{array}{*{20}{c}}
	{D_{1,1}^{ - 1}}&0\\
	{ - D_{2,2}^{ - 1}{D_{2,1}}D_{1,1}^{ - 1}}&{D_{2,2}^{ - 1}}
	\end{array}} \right]{S_n}\left[ {\begin{array}{*{20}{c}}
	(D_{1,1}')^{ - 1}&- ({D_{1,1}'})^{ - 1}{D_{2,1}'}(D_{2,2}')^{ - 1}\\
	0&(D_{2,2}')^{ - 1}
	\end{array}} \right],
\end{align}
with $\theta=\bar{\theta}$ in evaluating all the matrices involved.
From the asymptotic normality result, for fixed large sample size $n$, we can approximate the covariance matrix of $\hat{\theta}_n$ by $D_n^{-1}S_n (D_n')^{-1}$. Note that the estimating function $F_{n,1}$ for the diffusion parameter is fixed, so the upper left block matrix in $D_n^{-1}S_n (D_n')^{-1}$, which can be seen as the approximate covariance matrix for $\hat{\theta}_{n,1}$, is fixed. Our aim is to find the weighting matrix $W$ that minimizes the lower right block matrix in $D_n^{-1}S_n (D_n')^{-1}$, the approximate covariance matrix for $\hat{\theta}_{n,2}$. The precise definition is given below. $F^*_{n,2}$ is the estimating function constructed using the weighting matrix $W^*$ in \eqref{eq: F2}.
\begin{definition} \label{defOpt}
Let $\mathcal{W}$ be the collection of all possible weighting matrix. Define $D_{2,1}^* = E[{\partial _{{\theta _1}}}F_{n,2}^*(\theta )]$, $D_{2,2}^* = E[{\partial _{{\theta _2}}}F_{n,2}^*(\theta )]$, $S_{1,2}^*=E[F_{n,1}(\theta_1)(F_{n,2}^*(\theta) )' ]$ and $S_{2,2}^*=E[F_{n,2}^*(\theta) (F_{n,2}^*(\theta))' ]$. $W^*$ is optimal within $\mathcal{W}$ if the lower right sub-matrix of $D_n^{-1}S_n (D_n')^{-1}$ is minimized,
	that is,
\begin{align}
& D_{2,2}^{ - 1}\left[ {\begin{array}{*{20}{c}}
	{ - {D_{2,1}}D_{1,1}^{ - 1}}&I
	\end{array}} \right]\left[ {\begin{array}{*{20}{c}}
	{{S_{1,1}}}&{{S_{1,2}}}\\
	{{S_{2,1}}}&{{S_{2,2}}}
	\end{array}} \right]\left[ {\begin{array}{*{20}{c}}
	{ - {{({D_{1,1}'})}^{ - 1}}{D_{2,1}'}}\\
	I
	\end{array}} \right]{({D_{2,2}'})^{ - 1}} \notag \\
&  - {(D_{2,2}^*)^{ - 1}}\left[ {\begin{array}{*{20}{c}}
	{ - D_{2,1}^*D_{1,1}^{ - 1}}&I
	\end{array}} \right]\left[ {\begin{array}{*{20}{c}}
	{{S_{1,1}}}&{S_{1,2}^*}\\
	{S_{2,1}^*}&{S_{2,2}^*}
	\end{array}} \right]\left[ {\begin{array}{*{20}{c}}
	{ - {{({D_{1,1}'})}^{ - 1}}(D_{2,1}^*)'}\\
	I
	\end{array}} \right]{({D_{2,2}^*}')^{ - 1}} \notag
\end{align}
is positive semi-definite for all $W\in \mathcal{W}$. Here all quantities without * correspond to using weighting matrix $W$. We assume that $D_{2,2}$ is invertible for all $W\in \mathcal{W}$.
\end{definition}

\begin{remark}
Our definition of the optimal weighting matrix is only concerned with the covariance matrix of $\theta_2$. The true definition of optimality would be to look at the full covariance matrix. We have tried to derive the true optimal weighting matrix, but it is very difficult to obtain an expression for it. This is why we adopted the modified definition of optimality given in Definition \ref{defOpt}. Such modification will not cause significant loss of statistical efficiency provided that the estimator $\hat{\theta}_{n,1}$ is close to $\bar{\theta}_1$, the true value of $\theta_1$. This is because the optimal weighting matrix derived under the modified definition would be asymptotically close to the true one.

While we can also apply weighting to all the moment conditions, including those for estimating $\theta_1$, the derivation of the optimal weighting matrix would be even harder, and it will certainly require more computations to calculate it. For these reasons, we only consider weighting the eigenfunction-based estimating functions. The numerical examples in Section \ref{sec:num} show that our approach delivers good results.
\end{remark}

We next provide an equivalent characterization of optimality.
\begin{proposition}\label{prop:optimality}
$W^*$ is optimal if and only if
	\begin{equation}
D_{2,2}^{ - 1}\left[ {\begin{array}{*{20}{c}}
	{ - {D_{2,1}}D_{1,1}^{ - 1}}&I
	\end{array}} \right]\left[ {\begin{array}{*{20}{c}}
	{{S_{1,1}}}&{S_{1,2}^*}\\
	{{S_{2,1}}}&{\tilde{S}'_{2,2} }
	\end{array}} \right]\left[ {\begin{array}{*{20}{c}}
	{ - {{(D_{1,1}')}^{ - 1}}(D_{2,1}^*)'}\\
	I
	\end{array}} \right]
	\end{equation}
	is a constant matrix for any $W\in \mathcal{W}$ where $\tilde{S}_{2,2}=E[F^*_2(\theta)F'_2(\theta)].$
\end{proposition}

In the following, $E$ refers to taking expectation with the stationary distribution $\gQ$ as the initial distribution (for notational simplicity, we dropped $\gQ$ in the subscript). Using \eqref{eq: F1} and \eqref{eq: F2-1}, it is straightforward to obtain that
\[\left[ {\begin{array}{*{20}{c}}
	D_{1,1}&0\\
	D_{2,1}&D_{2,2}
	\end{array}} \right] = \left[ {\begin{array}{*{20}{c}}
	{nE[\partial_{\theta_1} f_1(Y_{t_0},Y_{t_1};\theta_1)]}&0\\
	{nE[W(Y_{t_0};\theta)'\partial_{\theta_1} V(Y_{t_0},Y_{t_1};\theta)]}&{nE[W(Y_{t_0};\theta)'\partial_{\theta_2} V(Y_{t_0},Y_{t_1};\theta)]}
	\end{array}} \right]\]
and
\[ S_{2,2} = nE\left[ {W(Y_{t_0};\theta)'V(Y_{t_0},Y_{t_1};\theta)V(Y_{t_0},Y_{t_1};\theta)'W(Y_{t_0};\theta)} \right].\]
Now we simplify $E[F_{n,1}(\theta_1)F{'_{n,2}}(\theta)]$. $\gF_t$ refers to the information generated by the process $Y$ up to time $t$. First, note that for $j<i$,
\begin{align}
&E[f_1({Y_{t_{j-1}}},Y_{t_j})V'({Y_{t_{i-1}}},{Y_{t_i}})W({Y_{t_{i-1}}})]= E[E[f_1({Y_{t_{j-1}}},Y_{t_j})V'({Y_{t_{i-1}}},{Y_{t_i}})W({Y_{t_{i-1}}})|{\gF_{t_{i-1}}}]]  \\
&=E[f_1({Y_{t_{j-1}}},Y_{t_j})E[V'({Y_{t_{i-1}}},{Y_{t_i}})|{\gF_{t_{i-1}}}]W(Y_{t_{i-1}})]=0.
\end{align}
Then,
\begin{align}
E[{F_{n,1}}F{'_{n,2}}]  &=E\left[\sum\limits_{j = 1}^n {\sum\limits_{i = 1}^n {f_1({Y_{t_{j-1}}},Y_{t_j})V'({Y_{t_{i-1}}},{Y_{t_i}})} } W({Y_{t_{i-1}}})\right] \notag \\
&= E\left[\sum\limits_{i = 1}^n {\sum\limits_{j \ge i}^n {f_1(Y_{t_{j-1}},{Y_{t_j}})V'({Y_{t_{i-1}}},{Y_{t_i}})} } W({Y_{t_{i-1}}})\right] \notag \\
& = \sum\limits_{i=1}^n {E\left[\sum\limits_{j \ge i}^n {f_1(Y_{t_{j-1}},{Y_{t_j}})V'({Y_{t_{i-1}}},{Y_{t_i}})W({Y_{t_{i-1}}})} \right]} \notag \\
&= \sum\limits_{i = 1}^n{E\left[\sum\limits_{j = 1}^{n - i + 1} {f_1(Y_{t_{j-1}},{Y_{t_j}})V'({Y_{t_0} },{Y_{t_1}})W({Y_{t_0}})} \right]} \notag \\
& = \sum\limits_{j = 1}^n {\sum\limits_{i = 1}^{n - j + 1} {E\left[ {f_1({Y_{{t_{j - 1}}}},{Y_{{t_j}}})V'({Y_{{t_0}}},{Y_{{t_1}}})W({Y_{{t_0}}})} \right]} } \notag \\
& = E\left[\left\{ {\sum\limits_{j = 1}^n {(n- j + 1)f_1\left( Y_{t_{j-1}},{{Y_{t_j}}} \right)} } \right\}V'({Y_{t_0}},{Y_{t_1}})W({Y_{t_0}})\right].
\end{align}
Define
\begin{equation}\label{eq:f1-tilde}
\tilde{f}_1\left(Y_{t_0}, Y_{t_1}\right): =  E\left[ \sum\limits_{j = 1}^n {(n - j + 1)f_1\left( Y_{t_{j-1}},{{Y_{t_j}}} \right)} \Bigg|Y_{t_0},Y_{t_1}\right].
\end{equation}
Then $E[F_{n,1}F{'_{n,2}}]=E[\tilde{f}_1(Y_{t_0},Y_{t_1})V'(Y_{t_0},Y_{t_1})W(Y_{t_0})]$. The optimality condition for our problem is that
\begin{align} \label {eq: optCon}
&D_{2,2}^{ - 1}\left[ {\begin{array}{*{20}{c}}
	{ - {D_{2,1}}D_{1,1}^{ - 1}}&I
	\end{array}} \right]\left[ {\begin{array}{*{20}{c}}
	{{S_{1,1}}}&{S_{1,2}^*}\\
	{{S_{2,1}}}&{\tilde{S}_{2,2} }
	\end{array}} \right]\left[ {\begin{array}{*{20}{c}}
	{ - {{(D_{1,1}')}^{ - 1}}(D_{2,1}^*)'}\\
	I
	\end{array}} \right] \notag \\
&=\frac{1}{n}E{\left[ {W'{\partial _{{\theta _2}}}V} \right]^{ - 1}} \{ nE\left[ {W'{\partial _{{\theta _1}}}V} \right]\left(D_{1,1}^{-1}S_{1,1}(D_{1,1}')^{-1}(D_{2,1}^*)'-D_{1,1}^{-1}S_{1,2}^* \right) \notag \\
&  - E[W'V\tilde{f}'](D_{1,1}')^{-1}(D_{2,1}^*)' +nE\left[W'VV'W^*\right] \}
\end{align}
is a constant matrix.
\eqref{eq: optCon} can be rewritten as
$E\left[W'H_1\right]^{-1}E\left[W'H_2\right]$,
where
\begin{align}
H_1(y)&=E\left[{\partial_{\theta_2} V}(Y_{t_0},Y_{t_1};\theta)|Y_{t_0}=y \right], \label{eq:H1}\\
H_2(y)&=E\left[{\partial_{\theta_1} V}(Y_{{t_0}},Y_{t_1};\theta)|Y_{t_0}=y \right](D_{1,1}^{-1}S_{1,1}(D_{1,1}')^{-1}(D_{2,1}^*)'-D_{1,1}^{-1}S_{1,2}^*) \\
&-\frac{1}{n}E\left[V(Y_{{t_0}},Y_{t_1};\theta)\tilde{f}_1'(Y_{t_0},Y_{t_1};\theta_1) |Y_{t_0}=y\right](D_{1,1}')^{-1}(D_{2,1}^*)' \\
&+ E\left[V(Y_{{t_0}},Y_{t_1};\theta)V'(Y_{{t_0}},Y_{t_1};\theta)|Y_{t_0}=y\right]W^*(y).\label{eq:H2}
\end{align}
Since \eqref{eq: optCon} is a constant for all $W$, we have
$E[W'H_2]=E[W'H_1]C$
where $C$ is a constant matrix. Since $W(y)$ is arbitrary, we can set $W(y)=O1_{(l,y_0)}(y)$ where $O$ is a constant matrix with each entry equal to 1 and $y_0$ is an arbitrary constant in $(l,r)$. Then, we have
\[\int_{l}^{y_0}OH_2(y)q(y)dy=\int_{l}^{y_0}OH_1(y)q(y)dyC. \]
Differentiating with respect to $y_0$ on both sides of the above equation, we get $H_2(y)=H_1(y)C$ for any $y\in (l,r)$. Thus, from \eqref{eq:H2},
\begin{align}
W^*(y)&=E\left[V(Y_{{t_0}},Y_{t_1})V'(Y_{{t_0}},Y_{t_1})|Y_{t_0}=y\right]^{-1} \notag \\
& \Big\{E\left[{\partial_{\theta_1} V}(Y_{t_0},Y_{t_1})|Y_{t_0}=y \right](D_{1,1}^{-1}S_{1,2}^*-D_{1,1}^{-1}S_{1,1}(D_{1,1}')^{-1}(D_{2,1}^*)') \notag \\
&+\frac{1}{n}E\left[V(Y_{{t_0}},Y_{t_1})\tilde{f}_1'(Y_{t_0},Y_{t_1}) |Y_{t_0}=y\right](D_{1,1}')^{-1}(D_{2,1}^*)' +E\left[{\partial_{\theta_2} V}(Y_{t_0},Y_{t_1})|Y_{t_0}=y \right]C \Big\}.
\end{align}
Since $D_{2,1}^*$ and $S_{1,2}^*$ involve $W^*$, the above equation is not an explicit expression for $W^*$. However, since $D_{1,1}^{-1}S_{1,2}^*-D_{1,1}^{-1}S_{1,1}(D_{1,1}')^{-1}(D_{2,1}^*)'$ and $(D_{1,1}')^{-1}(D_{2,1}^*)'$ do not depend on $y$, $W^*$ is of the following form
\begin{align} \label{eq: optW}
{W^*}({y}) &= {\left( {E[V({Y_{t_0}},{Y_{t_1}};\theta)V'({Y_{t_0}},{Y_{t_1}};\theta)|{Y_{t_0}=y}]} \right)^{ - 1}}E[\partial_{\theta_2} V({Y_{t_0}},{Y_{t_1}};\theta)|{Y_{t_0}=y}]{C_1} \\
&+{\left( {E[V({Y_{t_0}},{Y_{t_1}};\theta)V'({Y_{t_0}},{Y_{t_1}};\theta)|{Y_{t_0}=y}]} \right)^{ - 1}}E[\partial_{\theta_1} V({Y_{t_0}},{Y_{t_1}};\theta)|{Y_{t_0}=y}]{C_2}  \\
&+{\left( {E[V({Y_{t_0}},{Y_{t_1}};\theta)V'({Y_{t_0}},{Y_{t_1}};\theta)|{Y_{t_0}=y}]} \right)^{ - 1}}E[V({Y_{t_0}},{Y_{t_1}};\theta)\tilde{f}_1'(Y_{t_{t_0}},Y_{ t_1};\theta_1)|{Y_{t_0}=y}]{C_3},
\end{align}
where $C_1,C_2$ and $C_3$ are constant matrices. Note that the optimal weighting matrix is not unique because $cW^*$ is still optimal for any $c\ne0$. Here, we will try to find the general form of an optimal weighting matrix.
\begin{proposition} \label{prop:optW}
	$W^*$ of form \eqref{eq: optW} is an optimal weight when $C_1=I$ and $C_2, C_3$ solve the following linear system.
	\begin{equation}
	\left[ {\begin{array}{*{20}{c}}
		{ - {Q_2}}&{{D_{1,1}}' - {Q_3}}\\
		{{D_{1,1}} - {Q_3}'}&{n{S_{1,1}} - {Q_5}}
		\end{array}} \right]\left[ {\begin{array}{*{20}{c}}
		{{C_2}}\\
		{{C_3}}
		\end{array}} \right] = \left[ {\begin{array}{*{20}{c}}
		{{Q_1}}\\
		{{Q_4}}
		\end{array}} \right]
	\end{equation}
	where
	\begin{align}
		&{{Q_1} = E\left[ {{{\left( {\partial_{\theta_1} V({Y_{t_0}},{Y_{t_1}})} \right)}^\prime }{{\left( {E[V({Y_{t_0}},{Y_{t_1}})V'({Y_{t_0}},{Y_{t_1}})|{Y_{t_0}}]} \right)}^{ - 1}}E[\left( {\partial_{\theta_2} V({Y_{t_0}},{Y_{t_1}})} \right)|{Y_{t_0}}]} \right],}\\
		&{{Q_2} = E\left[ {{{\left( {\partial_{\theta_1} V({Y_{t_0}},{Y_{t_1}})} \right)}^\prime }{{\left( {E[V({Y_{t_0}},{Y_{t_1}})V'({Y_{t_0}},{Y_{t_1}})|{Y_{t_0}}]} \right)}^{ - 1}}E[\left( {\partial_{\theta_1} V({Y_{t_0}},{Y_{t_1}})} \right)|{Y_{t_0}}]} \right],}\\
		&{{Q_3} = E\left[ {{{\left( {\partial_{\theta_1} V({Y_{t_0}},{Y_{t_1}})} \right)}^\prime }{{\left( {E[V({Y_{t_0}},{Y_{t_1}})V'({Y_{t_0}},{Y_{t_1}})|{Y_{t_0}}]} \right)}^{ - 1}}E[V({Y_{t_0}},{Y_{t_1}})\tilde{f}_1'(Y_{t_0},Y_{t_1}) |{Y_{t_0}}]} \right],}\\
		&{{Q_4} = E\left[ {\tilde{f}_1(Y_{t_0},Y_{t_1})V'({Y_{t_0}},{Y_{t_1}}){{\left( {E[V({Y_{t_0}},{Y_{t_1}})V'({Y_{t_0}},{Y_{t_1}})|{Y_{t_0}}]} \right)}^{ - 1}}E[\left( {\partial_{\theta_2} V({Y_{t_0}},{Y_{t_1}})} \right)|{Y_{t_0}}]} \right],}\\
		&{{Q_5} = E\left[ {\tilde{f}_1(Y_{t_0},Y_{t_1})V'({Y_{t_0}},{Y_{t_1}}){{\left( {E[V({Y_{t_0}},{Y_{t_1}})V'({Y_{t_0}},{Y_{t_1}})|{Y_{t_0}}]} \right)}^{ - 1}}E[V({Y_{t_0}},{Y_{t_1}})\tilde{f}_1'(Y_{t_0},Y_{t_1})|{Y_{t_0}}]} \right].}
	\end{align}
\end{proposition}
In general, to compute $Q_1$ to $Q_5$ in closed-form is very difficult, even when analytical expressions for the eigenvalues and the eigenfunctions are available. To calculate them numerically also requires extensive computations.
In Section \ref{sec:num}, we numerically compute them in the problem of estimating the SubOU process. We then compare the standard error for each subordinator parameter using the optimal weighting matrix and the KS weighting matrix. The comparison reveals little difference between these two choices. Since the KS weighting matrix is much easier to compute, we will use it instead of the optimal one in our method.

\subsection{Numerical Approximations} \label{approEigenSec}
For estimating diffusions using eigenfunction based estimating functions, \cite{kessler1999estimating} only considers the case in which explicit expressions for the eigenvalues and the eigenfunctions are available. In general, they are not known in closed-form.

In this paper, we apply an efficient numerical method to compute the eigenvalues and the eigenfunctions accurately for the Sturm-Liouville (SL) problem associated with the given diffusion. A particularly attractive class of methods for solving the SL problem numerically is the coefficient approximation method (see \cite{pryce1993numerical}). Here, we use a particular type of coefficient approximation, known as constant perturbation method (CPM) with high-order corrections to achieve high-level of accuracy (see \cite{ledoux2004cp,ledoux2010solving}). This method can handle a large class of SL problems even with discontinuity in the coefficients. It is implemented in a Matlab package called MATSLISE, which we directly use in our implementation. In general, the accuracy of eigenvalues and eigenfunctions deteriorates as their order increases. Fortunately, we do not need to use a large number of eigenpairs in the estimating function \eqref{eq: F2-1}. Our numerical experiment in Section \ref{sec:num} shows that using only the first several eigenpairs suffices for statistical efficiency, which is in line with the finding of \cite{kessler1999estimating} for estimating diffusions.

In our estimation procedure, we first estimate the diffusion parameters using estimating function \eqref{eq: F1}. After they are obtained, we only need to run the CPM once to numerically calculate $(\lambda_i,\varphi_i(x))$ for $i$ from $1$ to $M$, because they only depend on the diffusion parameters. By separating the estimation of diffusion and subordinator parameters, we avoid running the CPM multiple times and thus making the estimation procedure computationally more efficient.

To run the CPM, we specify a finite grid $\Pi_1$ that covers a large enough region. The MATSLISE program returns approximations for the eigenvalues and for the eigenfunctions on the grid. To obtain an approximate value for an eigenfunction at a non-grid point, we use linear interpolation. Denote the $i$-th approximated eigenpair as $(\lambda^A_i,\varphi^A_i(x))$. Then, we approximate the original estimating function $F_{n,2}$ in \eqref{eq: F2-1} by the following
\begin{equation}\label{eq: F2-Approx}
F^A_{n,2}(\theta)=\sum_{i=1}^{n}(W^A_{KS}(y_{t_{i-1}};\theta))'V^A(y_{t_{i-1}},y_{t_i};\theta),
\end{equation}
where $ {V^A}(Y_t,Y_{t+\Delta};\theta)= {\varphi}^A_i(Y_{t+\Delta};\theta_1)-e^{-\phi(- {\lambda}^A_i;\theta)} {\varphi}^A_i(Y_t;\theta_1)$ and $W^A_{KS}$ is the approximated KS weighting matrix which solves
\[ P^A(y;\theta)W^A_{KS}(y;\theta)=Q^A(y;\theta).\]
Here,
\[{ {Q}^A_{i,j}}(y;\theta) = \frac{\partial }{{\partial {\theta_{2,j}}}}{e^{ - \phi ( - { \lambda^A _i};\theta )t}}{ {\varphi}^A _i}(y;\theta_1),\]
which can be calculated analytically as we know the Laplace exponent $\phi(\cdot)$ in closed-form. $P^A(y;\theta)$ is defined as in \eqref{eq:P-mat} by using the approximated eigenvalues and eigenfunctions.
To evaluate $P^A(y;\theta)$, we need to calculate the integral
\begin{equation}\label{eq:integral}
\int_l^r \varphi^A_i(x;\theta_1)\varphi^A_j(x;\theta_1)p^\phi(t,y,x;\theta)dx,
\end{equation}
which is equivalent to pricing an European option with payoff function $\varphi^A_i(x)\varphi^A_j(x)$ in a subordinate diffusion model. Recently, \cite{LiZhangSIAM} developed an efficient algorithm for pricing European options in general subordinate diffusion models. Their method requires specifying a grid to discretize the state space. In our implementation, we use a uniform grid $\Pi_2$ with step size $h$, although non-uniform grids can be used in Li and Zhang's method. We denote the approximation to $P^A(y;\theta)$ using their method by $P^h(y;\theta)$ and the resulting weighting matrix and the estimating function by $W^h_{KS}$ and $F^h_{n,2}$. In general, $\Pi_1$ and $\Pi_2$ can be different. Inaccuracy in the eigenpairs can cause significant loss of precision in the estimator. Therefore, in our implementation, we choose a fine $\Pi_1$ for the CPM, which guarantees high level of accuracy in the eigenvalues and the eigenfunctions. $\Pi_2$ does not need to be as fine as $\Pi_1$ and we choose it to be a sub-grid of $\Pi_1$. The error of using $F^h_{n,2}$ to approximate the exact $F_{n,2}$ is dominated by the error in calculating the integral \eqref{eq:integral}, which is $O(h^2)$ by \cite{LiZhangErr} (the approximation error for the eigenvalues and the eigenfunctions is at much higher order than $O(h^2)$ because the CPM is used with high-order corrections). The error order for approximating \eqref{eq:integral} can be further improved using extrapolation as pointed out in \cite{LiZhangErr}. Using two rather coarse grids $\Pi_{2}$ and $\Pi'_{2}$, one can extrapolate the results from these two grids to reduce the error to $O(h^3)$. We first calculate $P^h(y;\theta)$ and $W^h_{KS}(y;\theta)$ for $y$ on the grid $\Pi_2$. To obtain $W^h_{KS}(y;\theta)$ at non-grid points, we apply linear interpolation.

%The error of $F^h_{n,2}$ is also $O(h^2)$.

\section{Consistency and Asymptotic Normality} \label{sec:CAN}
The subordinate diffusion parameter space is denoted by $\Theta$, which is assumed to be an open subset of $\mathbb{R}^p$. $\bar{\theta}$ is the true value of $\theta$. Let $\bar Q(x,y)=q(x;\bar \theta)p^\phi(\Delta,x,y;\bar \theta)$, which is the joint density of $(Y_0,Y_\Delta)$ under the true parameter value if the initial density is the stationary one. $\bar E$ denotes taking expectation under the true parameter value $\bar\theta$. Recall the vector functions $f_1$ and $f_2$ in \eqref{eq: F1} and \eqref{eq: F2}. Let $f(x,y;\theta)=(f_1(x,y;\theta)',f_2(x,y;\theta)')'$, and $F_n(\theta)=(F_{n,1}(\theta)',F_{n,2}(\theta)')'$. In our analysis, we consider an arbitrary weighting matrix $W$ for $F_{n,2}$. The following assumption is imposed (similar assumptions are also made in \cite{kessler1999estimating} and \cite{sorensen1999asymptotics}).

%\begin{assumption}\label{assump:misc}
%(a) $f_1(x;\theta_1)$ is twice continuously differentiable with respect to $\theta_1$ and $x$. $f_2(x,y;\theta_1,\theta_2)$ is twice continuously differentiable with respect to $x, y, \theta_1$ and $\theta_2$.

%\noindent(b) Each element of $f_1$ and its first and second partial derivatives w.r.t. $\theta_1$, as well as each element of $f_2$ and its first and second partial derivatives w.r.t. $\theta$ are locally dominated integrable w.r.t. $\bar Q$. (A function $g(x,y;\theta)$ is locally dominated integrable w.r.t. $\bar Q$ if for each $\theta \in \Theta$, there exists a neighborhood $U(\theta)$ and a non-negative $\bar Q$ integrable function $h_{\theta}(x,y)$ such that $|g(x,y;\theta')|<h_{\theta}(x,y)$ for all $(x,y,\theta')\in (l,r)\times (l,r)\times U(\theta)$; see \cite{KesslerLindnerSorensen12}, p.5.).

%\noindent(c) Each element of $f_1(x,y;\theta_1)$ and $f_2(x,y;\theta)$ is square integrable w.r.t. $\bar Q$ for all $\theta \in \Theta$.

%\noindent(d) Let $A_1(\bar \theta_1)=\bar E[\partial_{\theta_1} f_1(Y_{t_0},Y_{t_1}; \theta_1)]$ and $A_2(\bar\theta)=\bar E[\partial_{\theta_2} f_2(Y_{t_0},Y_{t_1}; \theta)]$. Both $A_1(\bar \theta_1)$ and $A_2(\bar\theta)$ are non-singular.
%\end{assumption}

\begin{assumption}\label{assump:misc}
	$\bar \theta \in \Theta$ and a neighbourhood B of $\bar \theta$ in $\Theta$ exists such that the following conditions hold.
	
	\noindent(a) $f_1(x;\theta_1)$ is continuously differentiable w.r.t. $\theta_1$ on B for all $x$. $f_2(x,y;\theta_1,\theta_2)$ is continuously differentiable w.r.t. $\theta_1$ and $\theta_2$ on B for all $x$ and $y$.
	
	\noindent(b) Each element of the first partial derivatives of $f_1$ w.r.t. $\theta_1$, as well as each element of the first partial derivatives of $f_2$ w.r.t. $\theta$ are dominated for all $\theta \in B$ by a function which is integrable w.r.t. $\bar Q$.
	
	\noindent(c) Each element of $f_1(x,y;\theta_1)$ and $f_2(x,y;\theta)$ is integrable w.r.t. $\bar Q$ for all $\theta \in B\backslash \{\bar\theta\}$, and square-integrable w.r.t. $\bar Q$ for $\theta = \bar \theta$.
		
	\noindent(d) Let $A_1(\bar \theta_1)=\bar E[\partial_{\theta_1} f_1(Y_{t_0},Y_{t_1}; \bar{\theta}_1)]$ and $A_2(\bar\theta)=\bar E[\partial_{\theta_2} f_2(Y_{t_0},Y_{t_1}; \bar{\theta})]$. $A_1(\bar \theta_1)$ and $A_2(\bar\theta)$ are non-singular.
\end{assumption}

A function $g(x,y;\theta)$ is locally dominated integrable w.r.t. $\bar Q$ if for each $\theta \in \Theta$, there exists a neighborhood $U(\theta)$ and a non-negative $\bar Q$ integrable function $h_{\theta}(x,y)$ such that $|g(x,y;\theta')|<h_{\theta}(x,y)$ for all $(x,y,\theta')\in (l,r)\times (l,r)\times U(\theta)$; see \cite{KesslerLindnerSorensen12}, p.5.
\begin{assumption} \label{assum: globalCons}
	
(a) $\bar E[ f(Y_{t_0},Y_{t_1}; \theta)]\ne 0$ for all $\theta \ne \bar \theta$.

\noindent(b) Each element of $f_1$ and  each element of $f_2$ are locally dominated integrable w.r.t. $\bar Q$.
\end{assumption}

\begin{remark}
When the weighting matrix is given by the KS one, it can be shown that $A_2(\bar\theta)$ is positive definite and hence non-singular as long as $M\ge p_2$.
\end{remark}

We next develop the Central Limit Theorem (CLT) for our estimating function. Since $F_{n,1}$ is not a martingale, to develop its CLT, we need the following property.
\begin{proposition}
Under Assumption \ref{assump:diff}, the diffusion transition operator $\gP_t$ is a strong contraction for any $t>0$, that is, there exists some $\delta>0$ such that $\|\gP_t f\|_2\le \exp(-\delta t)\|f\|_2$ for any $f\in L^2(I,q)$ such that $\int_I f(x)q(x)dx=0$ ($\|\cdot\|_2$ denotes the $L^2(I,q)$ norm). $\gP^\phi_t$ is also a strong contraction for any $t>0$.
\end{proposition}
\begin{proof}
For any $f\in L^2(I,q)$, $\gP_tf$ admits an eigenfunction expansion
\begin{equation}
\gP_t f(x) = \sum_{i=0}^\infty f_i e^{\lambda_i t}\varphi_i(x),\ f_i=\int_I f(x)\varphi_i(x)q(x)dx.
\end{equation}
From \eqref{eq:0-th-eigen}, $\lambda_0=0$, $\varphi_0(x)\equiv 1$. So for $f$ such that $\int_I f(x)q(x)dx=0$, $f_0=0$. Using the orthonormality of $\{\varphi_i(x), i=1,2,\cdots\}$ and that $0>\lambda_1>\lambda_i$ for $i>1$, we obtain
\begin{equation}
\|\gP_t f\|_2^2 = \sum_{i=1}^\infty f_i^2 e^{2\lambda_i t}< e^{2\lambda_1 t}\sum_{i=1}^\infty f_i^2 = e^{2\lambda_1 t} \|f\|_2^2.
\end{equation}
Thus $\gP_t$ is a strong contraction. For $\gP^\phi_t$, using the definition of subordination, for $f\in L^2(I,q)$ such that $\int_I f(x)q(x)dx=0$,
\begin{align}
\|\gP^\phi_tf\|_2&=\left\|\int_{(0,\infty)}\gP_uf s_t(du)\right\|_2\le \int_{(0,\infty)}\|\gP_uf \|_2s_t(du)\\
&\le \int_{(0,\infty)}\exp(-\delta u)\|f\|_2s_t(du)=e^{-\phi(\delta)t}\|f\|_2.
\end{align}
Here, $s_t(du)$ is the distribution of $T_t$. Since $\phi(\delta)>0$, $\gP^\phi_t$ is also a strong contraction.
\end{proof}

Now we can investigate the convergence of $F_n(\theta)/\sqrt{n}$.

\begin{proposition} \label{prop:CLT}
Under Assumption \ref{assump:diff} and \ref{assump:misc},	
\begin{equation} \label{eq: CLT}
\frac{1}{{\sqrt n }}{F_n}({\bar \theta}) \to N(0,\Sigma(\bar{ \theta}) ), \quad \Sigma(\bar{ \theta})  = \left[ {\begin{array}{*{20}{c}}
	{{\Sigma _1(\bar{ \theta})}}&{{\Sigma _2(\bar{ \theta})}}\\
	{{\Sigma '_2(\bar{ \theta})}}&{{\Sigma _3(\bar{ \theta})}}
	\end{array}} \right],
\end{equation}
where
\begin{align}
{\Sigma_1}(\bar{ \theta}) &= \bar E\left[ {{f_1}({Y_{t_0}},{Y_{t_1}};{\bar{ \theta}}){f_1}({Y_{t_0}},{Y_{t_1}};{\bar{ \theta}})'} \right] + \bar E\left[ {{f_1}({Y_{t_0}},{Y_{t_1}};{\bar{ \theta}})\left( {{{(I- \gP_{\Delta}^\phi )}^{ - 1}}{\mathfrak{f}_1}({Y_{t_0}};{\bar{ \theta}})} \right)'} \right]   \\
&+ \bar E\left[ {\left( {{{(I- \gP_{\Delta}^\phi)}^{ - 1}}{\mathfrak{f}_1}({Y_{t_0}};{\bar{ \theta}})} \right){f_1}({Y_{t_0}},{Y_{t_1}};{\bar{ \theta}})'} \right],\label{eq: Sigma1}\\
{\Sigma_2}(\bar{\theta}) &=\bar E\Big[ \Big( f_1(Y_{t_0},Y_{t_1};\bar{\theta}) - \mathfrak{f}_1(Y_{t_0};\bar{\theta}) + (I- \gP_{\Delta}^\phi )^{-1}\mathfrak{f}_1(Y_{t_1};\bar{\theta})-\gP_{\Delta}^\phi(I- \gP_{\Delta}^\phi )^{ - 1}\mathfrak{f}_1(Y_{t_0};\bar{\theta}) \Big)\\
&f_2(Y_{t_0},Y_{t_1};\bar{\theta})' \Big],\label{eq: Sigma2}\\
{\Sigma _3}(\bar{ \theta}) &=\bar E\left[ {{f_2}({Y_{t_0}},{Y_{t_1}};{\bar{ \theta}}){f_2}({Y_{t_0}},{Y_{t_1}};{\bar{ \theta}})'} \right],\label{eq: Sigma3}
\end{align}
where ${\mathfrak{f}_1}(y;{\bar \theta}) = \bar E\left[ {{f_1}({Y_{t_0}},{Y_{t_1}};{\bar \theta})|{Y_{t_0}} = y} \right]$.
\end{proposition}

\begin{proof}
The strong contractioness of $\gP^\phi_t$ guarantees the existence of $(I-\gP^\phi_t)^{-1}$. $F_{n,1}$ is not a martingale. To develop CLT for it, we can adapt the arguments in \cite{hansen1995back}, p.797-798. $F_{n,2}$ is a martingale, so we can apply the martingale CLT (\cite{billingsley1961lindeberg}, Theorem 1). The details are omitted here.
\end{proof}

Based on Proposition \ref{prop:CLT}, we have the following consistency and asymptotic normality result by applying Theorem 1.2.2 of \cite{sorensen2012estimating}.
\begin{thm} \label{CAN1}
Under Assumption \ref{assump:diff} and \ref{assump:misc}, an estimator $\hat{\theta}_n$ exists with a probability tending to one as $n\to \infty$.
 Moreover,
\[\hat{\theta}_n \overset{p}{\to} \bar{ \theta} ,\]
and
\[\sqrt{n}(\hat{\theta}_n - \bar{ \theta}) \to \mathcal {N}(0,\mathcal{V}), \]
where $\mathcal{V}=A(\bar{ \theta})^{-1}\Sigma(\bar{ \theta})(A(\bar{ \theta})^{-1})'$ and $A(\bar \theta)=\bar E[\partial_\theta f(Y_{t_0},Y_{t_1};\bar{\theta})]$ (the invertibility of $A(\bar \theta)$ is guaranteed by Assumption \ref{assump:misc} (d)). Moreover, under Assumption \ref{assum: globalCons}, the estimator $\hat{\theta}_n$ is the unique $F_n$-estimator on any bounded subset of $\Theta$ containing $\bar \theta$ with probability tending to 1 as $n \to \infty$.
\end{thm}

Theorem \ref{CAN1} does not consider that $\hat{\theta}_{n,2}$ generally cannot be computed exactly. Next, we take into consideration the effect of numerical approximations that are used in Section \ref{approEigenSec} to compute the weighting matrix in the KS way. In the following, $\|v\|$ is the Euclidean norm of vector $v$, and for matrix $A$, $\|A\|:=\sqrt{\text{tr}(AA')}$. Recall that $h$ is the step size of the grid $\Pi_2$.
The computation of $\hat{\theta}_{n,1}$ does not require discretization. For the estimator of $\theta_2$, we write it as $\hat{\theta}^h_{n,2}$ because it depends on the grid that is used. We put $F^h_n(\theta)=(F_{n,1}(\theta)',F^h_{n,2}(\theta)')'$ and let $J_n^h(\theta)=\partial_\theta F_n^h(\theta)$. Suppose that
\begin{align}
&\|F_{n,2}^h(\theta)-F_{n,2}(\theta)\|\le C(\theta) n h^p\ \text{almost surely},\label{eq:error_estfunc}\\
&|J_n^h(\theta)_{i,j}-J_n(\theta)_{i,j}|\le C_{ij}(\theta) n h^q\ \text{almost surely\ for}\ p_1<i\le p_1+p_2, 1\le j\le p_1+p_2, \label{eq:error_estfucderiv}
\end{align}
for some $p,q>0$, and $C(\theta)$, $C_{ij}(\theta)$ are continuous with respect to $\theta$. In \eqref{eq:error_estfunc}, $p=2$ without extrapolation and $p=3$ with extrapolation in view of the discussions in Section \ref{approEigenSec}. As $F_n^h(\theta)_i=F_n(\theta)_i$ for $1\le i\le p_1$, \eqref{eq:error_estfunc} implies that
\begin{equation}\label{eq:error_estfunc1}
\|F_{n}^h(\theta)-F_{n}(\theta)\|\le C(\theta) n h^p\ \text{almost surely}.
\end{equation}
In general, the step size $h$ depends on the number of observations $n$, so we will write it as $h_n$ below whenever necessary. Our main result is that previous conclusions about consistency and asymptotic normality hold under suitable assumptions on the convergence of $h_n$. We need the next two lemmas.

%\begin{lemma} \label{JAlemma1}
%Under Assumption \ref{assump:diff} and \ref{assump:misc} (a) and (b), assume $\lim\limits_{n\to \infty} h_n=0.$ Then for any compact subset $S$ of $B$ (defined in Assumption \ref{assump:misc})
%	\[ \sup_{\theta \in S}|n^{-1}J_n^h(\theta)_{i,j}-A({ \theta})_{i,j}| \to 0 \]
%	almost surely. Here $A(\theta)=\bar E[\partial_\theta f(Y_{t_0},Y_{t_1};\theta)]$.
%\end{lemma}
%\begin{proof}
%On the compact subset $S$, there exists a finite number $K$ such that $|C_{ij}(\theta)|\le K$ for $\theta \in S$ since $C_{ij}(\theta)$ is continuous. Then
%	\begin{align}
%	& \sup_{\theta \in S }|n^{-1}J_n^h(\theta)_{i,j}-A({ \theta})_{i,j}|\\
%	&\le \sup_{\theta \in S}|n^{-1}J_n(\theta)_{i,j}-A({ \theta})_{i,j}|+\sup_{\theta \in S}|n^{-1}J_n^h(\theta)_{i,j}-n^{-1}J_n({ \theta})_{i,j}|\\
%	&\le \sup_{\theta \in S}|n^{-1}J_n(\theta)_{i,j}-A({ \theta})_{i,j}|+\sup_{\theta \in S}|C(\theta)_{i,j}|h_n^q\\
%	&\le \sup_{\theta \in S}|n^{-1}J_n(\theta)_{i,j}-A({ \theta})_{i,j}|+K h_n^q
%	\end{align}
%	The first term converges to $0$ almost surely by Lemma 1.2.3 in \cite{sorensen2012estimating}. Obviously the second term also converges to $0$.
%\end{proof}

\begin{lemma} \label{JAlemma}
	Let $J_n(\theta)=\partial_\theta F_n(\theta)$ (recall the notation introduced at the end of Section \ref{Sec:intro}).  Let $\epsilon_n>0$ be a decreasing sequence and $\lim_{n\to\infty}\epsilon_n=0$. Under Assumption \ref{assump:diff} and \ref{assump:misc}, we have
	\[ \sup_{\theta \in \bar B_{\epsilon_n} (\bar{ \theta})}|n^{-1}J_n(\theta)_{i,j}-A(\bar{ \theta})_{i,j}| \to 0 \]
	almost surely as $n \to \infty$, where  $\bar B_{\epsilon_n} ({\bar{ \theta}}) = \{ \theta \in \Theta:\|\theta - {\bar \theta}\| \le \epsilon_n \} $.
\end{lemma}
\begin{proof}
	Note that
	\begin{align}
		& \sup_{\theta \in \bar B_{\epsilon_n} (\bar{ \theta})}|n^{-1}J_n(\theta)_{i,j}-A(\bar{ \theta})_{i,j}|\\
		&\le \sup_{\theta \in \bar B_{\epsilon_n} (\bar{ \theta})}|n^{-1}J_n(\theta)_{i,j}-A({ \theta})_{i,j}|+\sup_{\theta \in \bar B_{\epsilon_n} (\bar{ \theta})}|A(\theta)_{i,j}-A(\bar{ \theta})_{i,j}|\\
		&\le \sup_{\theta \in \bar B_{\epsilon_1} (\bar{ \theta})}|n^{-1}J_n(\theta)_{i,j}-A({ \theta})_{i,j}|+\sup_{\theta \in \bar B_{\epsilon_n} (\bar{ \theta})}|A(\theta)_{i,j}-A(\bar{ \theta})_{i,j}|
	\end{align}
	Since $\bar B_{\epsilon_1} (\bar \theta)$ is a compact set, the first term converges to $0$ almost surely by Lemma 1.2.3 in \cite{sorensen2012estimating}. As $\bar B_{\epsilon_n}(\bar \theta) \to \{\bar \theta \}$, the second term also converges to 0.
\end{proof}

\begin{lemma} \label{FhLemma}
Under Assumption \ref{assump:diff}, \ref{assump:misc}, \ref{assum: globalCons} (b), suppose $\lim_{n\to \infty}h_n=0$. Then, for any compact subset $S$ of $\Theta$, we have
\begin{equation}
\sup_{\theta \in S}||n^{-1}F^h_n(\theta)-\bar E[ f(Y_{t_0},Y_{t_1}; \theta)]|| \overset{p}{\to} 0,\label{eq:conv-Fhn}\\
\end{equation}
and for any compact subset $S$ of $B$ (defined in Assumption \ref{assump:misc}),
\begin{equation}
\sup_{\theta \in S}|n^{-1}J_n^h(\theta)_{i,j}-A({ \theta})_{i,j}| \overset{p}{\to} 0,\label{eq:conv-Jhn}
\end{equation}
for each $ij$-th entry of $J_n^h(\theta)$. Here, $A(\theta)=\bar E[\partial_\theta f(Y_{t_0},Y_{t_1};\theta)]$.
\end{lemma}
\begin{proof}
For any compact subset $S$ of $\Theta$, there exists a finite number $K$ such that $|C(\theta)|\le K$ for $\theta \in S$ since $C(\theta)$ is continuous. Then,
\begin{align}
& \sup_{\theta \in S}||n^{-1}F^h_n(\theta)-\bar E[ f(Y_{t_0},Y_{t_1}; \theta)]|| \\
& \le \sup_{\theta \in S}||n^{-1}F_n(\theta)-\bar E[ f(Y_{t_0},Y_{t_1}; \theta)]|| +\sup_{\theta \in S}||n^{-1}F^h_n(\theta)- n^{-1}F_n(\theta)|| \\
& \le \sup_{\theta \in S}||n^{-1}F_n(\theta)-\bar E[ f(Y_{t_0},Y_{t_1}; \theta)]|| +\sup_{\theta \in S}||n^{-1}C(\theta)||nh_n^p \\
& \le \sup_{\theta \in S}||n^{-1}F_n(\theta)-\bar E[ f(Y_{t_0},Y_{t_1}; \theta)]|| +Kh_n^p.
\end{align}
By the finite covering property of a compact set, the dominated integrability condition of Lemma 1.2.3 in \cite{sorensen2012estimating} follows from the local dominated integrability of $f$ (Assumption \ref{assum: globalCons} (b)) and the continuity of $f$ (Assumption \ref{assump:misc} (a)). Applying this lemma shows the convergence of the first term. Obviously, the second term converges to $0$. The second claim can be proved similarly.
\end{proof}
\color{black}

\begin{thm} \label{CAN2}
	Under Assumption \ref{assump:diff} and \ref{assump:misc}, and \eqref{eq:error_estfunc}, \eqref{eq:error_estfucderiv}, suppose that $F_n^h(\theta)$ is continuously differentiable in a neighbourhood $B'$ of $\bar \theta$ in $\Theta$ and  $\lim_{n\to \infty}h_n =0$. Then, an estimator $\hat{\theta}_n^h$ exists with probability tending to 1 as $n\to \infty$ and
	\[\hat \theta_{n,2}^h \overset{p}{\to} \bar \theta_2 . \]
	on the set where they exist. Under Assumption \ref{assum: globalCons}, the estimator $\hat \theta_{n,2}^h$ is the unique $F^h_n$-estimator on any bounded subset of $\Theta$ containing $\bar \theta$ with probability tending to 1 as $n \to \infty$. Moreover, for $n$ large enough, $\|\hat{\theta}_{n,2}^h-\hat{\theta}_{n,2}\|=O(h_n^p)$ almost surely on the set where they exist.
\end{thm}
\begin{proof}
	We check the three conditions required for Theorem 1.10.2 in \cite{sorensen2012estimating}. Condition (i) follows from \eqref{eq:conv-Fhn}. Let $M$ be a compact subset of $B\cap B'$. \eqref{eq:conv-Jhn} implies condition (ii). Condition (iii) is equivalent to Assumption \ref{assump:misc} (d). Thus, applying this theorem shows the existence of a consistent estimator $\hat \theta_{n,2}^h$.

	Now we prove the second statement. Let $\bar{B}_{\epsilon} (\bar \theta)$ denote the closed ball with radius $\epsilon$ centered at $\bar \theta$. By Assumption \ref{assum: globalCons} (a), for any bounded subset $S$ of $\Theta$ containing $\bar{\theta}$, we have
	\begin{equation}
		\bar{P}\left(\inf_{S \backslash \bar{B}_{\epsilon}(\bar \theta_2) } \|\bar E[ f(Y_{t_0},Y_{t_1}; \theta)]\| >0\right)=1
	\end{equation}
	for all $\epsilon>0$. This, together with Lemma \ref{FhLemma}, implies that the conditions of Theorem 1.10.3 in \cite{sorensen2012estimating} are satisfied. Thus, for any sequence $\tilde \theta_{n,2}^h$ of $F_n^h$-estimator,
	\begin{equation} \label{eq: globalCon}
		\bar{P}(\tilde \theta_{n,2}^h \in S \backslash \bar{B}_\epsilon(\bar \theta_2) ) \to 0
	\end{equation}
	as $n \to \infty$ for all $\epsilon>0$. Let $\theta'_{n,2}$ be an $F_n^h$-estimator. Let $\theta''_{n,2}=\theta'_{n,2} 1( \theta'_{n,2} \in S)+\hat \theta_{n,2}^h1( \theta'_{n,2} \notin S)$ where $\hat \theta_{n,2}^h$ is a known consistent $F_n^h$-estimator. Thus,
	$ \theta''_{n,2}$ is a consistent estimator by \eqref{eq: globalCon}. Then, by Theorem 1.10.2 in \cite{sorensen2012estimating}, $P(\theta''_{n,2} \ne \hat \theta_{n,2}^h) \to 0$ as $n\to \infty$ which means $\hat \theta_{n,2}^h$ is eventually unique on $S$.
	
	For the last part, due to the consistency of $\hat \theta_{n}$ and $\hat \theta_{n}^h$, there exists a sequence
	$\epsilon_n>0$ satisfying $\epsilon_n\to 0$ and $\bar P(E_n)\to 1$, where
	\begin{equation} \label{eq: Event}
		E_n=\{\hat \theta_{n}, \hat \theta_{n}^h \in \bar B_{\epsilon_n}(\bar \theta) \} .
	\end{equation}
	Applying the mean value theorem to each element of $F_{n,2}(\hat\theta_{n,1},\hat \theta_{n,2}^h)$ (the $i$-th element is denoted by $F_{n,2}^i(\hat \theta_{n,1}, \theta_2)$) on $E_n$, we have
	\begin{align}
		F_{n,2}(\hat \theta_{n,1}, \hat \theta_{n,2}^h) - F_{n,2}^h(\hat \theta_{n,1}, \hat \theta_{n,2}^h) &= F_{n,2}^{}(\hat \theta_{n,1}, \hat \theta_{n,2}^h) \notag \\
		& = F_{n,2}(\hat \theta _{n,1},\hat \theta _{n,2}) + J_{n,2}(\beta^1_{n},\cdots ,\beta^{p_2}_n)(\hat \theta _{n,2}^h - \hat \theta _{n,2})\\
		& =J_{n,2}(\beta^1_{n},\cdots ,\beta^{p_2}_n)(\hat \theta _{n,2}^h - \hat \theta _{n,2}),
	\end{align}
	where $\beta^i_n=a_i(\hat\theta_{n,1}',(\hat{\theta}_{n,2}^{h})')'+(1-a_i)(\hat\theta_{n,1}',\hat\theta_{n,2}')'$ for some $0<a_i<1$, and
	for $i=1,\cdots, p_2$,
	
	\[ {J_{n,2}}(\beta ^1,\cdots,\beta ^{{p_2}}) = \left[ {\begin{array}{*{20}{c}}
		\partial _{\theta_2}{F_{n,2}^1}(\beta^1)\\
		{\cdots}\\
		\partial _{\theta_2}{F_{n,2}^{p_2} }(\beta^{p_2})
		\end{array}} \right].\]
	Thus
	\begin{align}
		\|\hat \theta _{n,2}^h - \hat \theta _{n,2}\| &= \|{\left( J_{n,2}(\beta^1_n,\cdots ,\beta^{p_2}_n)\right)^{ - 1}} \left( F_{n,2}(\hat \theta_{n,1}, \hat \theta_{n,2}^h) - F_{n,2}^h(\hat \theta_{n,1}, \hat \theta_{n,2}^h)  \right)\| \\
		&= \|{\left( J_{n,2}(\beta^1_n,\cdots ,\beta^{p_2}_n)/n \right)^{ - 1}} \left(  F_{n,2}(\hat \theta_{n,1}, \hat \theta_{n,2}^h) - F_{n,2}^h(\hat \theta_{n,1}, \hat \theta_{n,2}^h) \right) /n \| \\
		&\le \|{\left( J_{n,2}(\beta^1_n,\cdots ,\beta^{p_2}_n)/n \right)^{ - 1}}\| C(\hat \theta_{n,1}, \hat \theta_{n,2}^h )h_n^p\ \text{almost surely.}\label{eq: difference}
	\end{align}
	$C(\hat \theta_{n,1}, \hat \theta_{n,2}^h )$ is bounded for all $n$ due to the continuity of $C(\theta)$ is continuous and the boundedness of $E_n$.  Since each $\beta^i_n\in \bar B_{\epsilon_n}(\bar \theta)$, by Lemma \ref{JAlemma}, $J_{n,2}(\beta^1_n,\cdots,\beta^{p_2}_n )/n\to A_2(\bar \theta)$ almost surely. So the norm in \eqref{eq: difference} is bounded for $n$ almost surely. Thus, $\|\hat \theta _{n,2}^h - \hat \theta _{n,2}\|=O(h_n^p)$ almost surely on the set where they exist.
\end{proof}

\begin{thm}
	Under the assumptions made in Theorem \ref{CAN2}, further assume that $\lim\limits_{n\to \infty}\sqrt{n}h_n^p =0$. We have
	\[\sqrt{n}(\hat{\theta}_n ^h- \bar{ \theta}) \to  \mathcal {N}(0,\mathcal{V}),  \]
	where $\mathcal{V}$ is defined in Theorem \ref{CAN1}.
\end{thm}

\begin{proof}
	We want to show that $\sqrt{n} (\hat{\theta}_{n,2}^h-\bar{ \theta}_2) -\sqrt{n}(\hat{\theta}_{n,2}-\bar{ \theta}_2 ) \overset{p}{\to}  0$ which implies the claim. From \eqref{eq: difference}, on $E_n$ defined by \eqref{eq: Event},
	\begin{align}
		&\|\sqrt{n} (\hat{\theta}_{n,2}^h-\bar{ \theta}_2) -\sqrt{n}(\hat{\theta}_{n,2}-\bar{ \theta}_2 )\|=\|\sqrt{n}(\hat{\theta}_{n,2}^h-\hat{\theta}_{n,2} )\|\\
		&\le\sqrt{n}\|{\left( J_{n,2}(\beta^1_n,\cdots ,\beta^{p_2}_n)/n \right)^{ - 1}}\| C( \hat \theta_{n,1}, \hat \theta_{n,2}^h)h_n^p\ \text{almost surely.}
		%& \to \|(A_2(\bar{ \theta}))^{-1} \|(\sqrt{n})^{-1}C( \hat \theta_{n,1}, \hat \theta_{n,2}^h )h_n^2\to 0
	\end{align}
	Applying Lemma \ref{JAlemma} and that $\bar P(E_n)\to 1$, we have $\sqrt{n} (\hat{\theta}_{n,2}^h-\bar{ \theta}_2) -\sqrt{n}(\hat{\theta}_{n,2}-\bar{ \theta}_2 ) \overset{p}{\to} 0$.
\end{proof}

\section{Numerical Examples}\label{sec:num}
In this section, we consider the IG-SubOU process $Y$ discussed in Example \ref{eg:IG-SubOU} with $\gamma=0$ (the subordinator has no drift and hence $Y$ is a pure jump process), which we will use to evaluate several important questions numerically in Section \ref{sec:comp-weight} to \ref{sec:DataFreq}. This is a non-trivial process for which many quantities can be computed analytically. As an application of our method, we will construct a subordinate diffusion model to fit VIX data in Section \ref{sec:VIX}.

For the IG-SubOU process, let $\mu$ and $v$ be the mean rate and variance rate of the inverse Gaussian subordinator, i.e., $\mu=E[T_1]$ and $v=\text{Var}[T_1]$. We use $\mu$ and $v$ instead of $C$ and $\eta$ to reparameterize the L\'{e}vy density of $T$, because $\mu,v$ are easy to interpret. Under the new parametrization,  $\phi(\lambda)=\frac{\mu^2}{\nu}\left[\sqrt{1+2\frac{\nu}{\mu}\lambda}-1\right]$. The parameters of the pure jump IG-SubOU process are $(\kappa,\vartheta,\sigma,\mu,v)$. The eigenpairs $\{(\varphi_m,\lambda _m):m=0,1,\cdots,\}$ of the OU process are given in \eqref{eq:OU-eigen}, which have explicit expressions. The computation of $\varphi_m(x)$ is particularly simple and can be done via the following recursion
\begin{equation}
\varphi_0(x)=1,\ \varphi_1(x)=\frac{\sqrt{2\kappa}}{\sigma}(x-\vartheta),\ \varphi_m(x) = \sqrt {\frac{2}{m}}\frac{\sqrt{\kappa}}{\sigma}(x-\vartheta)\varphi_{m- 1}(x) - \sqrt{\frac{m-1}{m}} \varphi_{m - 2}(x),\ m\ge 2.
\end{equation}
The true parameter values for $Y$ are given by $\bar{\kappa}=0.04$, $\bar{\vartheta}=0$, $\bar{\sigma}=0.06$, $\bar{\mu}=1$ and $\bar{v}=0.5$ (the time unit is day). To estimate the parameters, we generate 2000 daily data by simulation (i.e., $\Delta=1$ day).

From the scaling invariance pointed out in Corollary \ref{cor:TSS}, we fix $\sigma$ and only estimate $(\kappa,\vartheta,\mu,v)$.
To estimate the diffusion parameter $(\kappa,\vartheta)$, we use the test function \eqref{eq:score-vec} in \eqref{eq: diffCon1} (the stationary density of the OU process is given in \eqref{eq:OUStationary}), and obtain two moment conditions
\begin{equation}
E_{\gQ}[(Y_t - \vartheta)^2 - \sigma^2/(2\kappa)]=0,\ E_{\gQ}[Y_t - \vartheta ]=0.
\end{equation}
Using these moment conditions to construct estimating functions, we obtain the estimator for $(\kappa,\vartheta)$ as follows (note that the value of $\sigma$ is fixed in advance)
\begin{equation}
\hat{\vartheta}=\frac{1}{n}\sum_{i=1}^n Y_{t_i},\ \hat{\kappa}=\frac{\sigma^2n}{2\sum_{i=1}^n(Y_{t_i}-\hat{\vartheta})^2}.
\end{equation}

\subsection{Comparison of the KS Weight and the Optimal Weight}\label{sec:comp-weight}
We use estimating function based on eigenfunctions to estimate $(\mu,v)$. To get the KS weight, we solve \eqref{eq: solveW}. For the IG-SubOU process, $P$ and $Q$ can be obtained in closed-form. We have
\begin{align}
& {Q_{i,1}}(y;\theta) = e^{ - \phi (\kappa i;\theta)\Delta}{\varphi_i}(y;\theta_1)\Delta \left[ {\frac{{2\mu }}{v } - \frac{{2\frac{\mu }{v } + 3\kappa i}}{{\sqrt {1 + 2\frac{v }{\mu }\kappa i} }}} \right], \\
& {Q_{i,2}}(y;\theta) = {e^{ - \phi (\kappa i;\theta)\Delta}}\varphi (y;\theta_1)\Delta \left[ { - {{\left( {\frac{\mu }{v }} \right)}^2} + \frac{{{{\left( {\frac{\mu }{v }} \right)}^2} + \frac{\mu }{v }\kappa i}}{{\sqrt {1 + 2\frac{v }{\mu }\kappa i} }}} \right].\label{eq: QofOU}
\end{align}
Using
\[{\varphi_i}(x){\varphi_j}(x) = \sum\nolimits_{r = 0}^{\min (i,j)} {\sqrt {\left( {\begin{array}{*{20}{c}}
			{i + j - 2r}\\
			{i - r}
			\end{array}} \right)\left( {\begin{array}{*{20}{c}}
			i\\
			r
			\end{array}} \right)\left( {\begin{array}{*{20}{c}}
			j\\
			r
			\end{array}} \right)} } {\varphi_{i + j - 2r}}(x),\]
we obtain that
\begin{align}
	{P_{i,j}}(y;\theta) &=\sum\nolimits_{r = 0}^{\min (i,j)} {\sqrt {\left( {\begin{array}{*{20}{c}}
					{i + j - 2r}\\
					{i - r}
				\end{array}} \right)\left( {\begin{array}{*{20}{c}}
				i\\
				r
			\end{array}} \right)\left( {\begin{array}{*{20}{c}}
			j\\
			r
		\end{array}} \right)} } {e^{ - \phi(\kappa(i + j - 2r);\theta)\Delta}}{\varphi_{i + j - 2r}}(y;\theta) \\
&- e^{ - (\phi (\kappa i;\theta)\Delta + \phi(\kappa j;\theta))\Delta }{\varphi_i}(y;\theta){\varphi_j}(y;\theta). \label{eq: PofOU}
\end{align}

To calculate the optimal weight, we need to solve the linear system in Proposition \ref{prop:optW} to get $C_2$ and $C_3$. We calculate $Q_1$ to $Q_5$ by Monte Carlo simulation with 20,000 replications. Note that the inner expectations in the expressions for $Q_1$ to $Q_5$, $E[V(Y_{t_0},Y_{t_1})V'(Y_{t_0},Y_{t_1})|Y_{t_0}=y]$, $E[\partial_{\theta_2} V(Y_{t_0},Y_{t_1})|Y_{t_0}=y]$, $E[V(Y_{t_0},Y_{t_1})\tilde{f}_1'|Y_{t_0}=y]$ and $E[\partial_{\theta_1} V(Y_{t_0},Y_{t_1})|Y_{t_0}=y]$ can all be computed in closed-form for the SubOU process (we do not show the formulas here to save space).

We next discuss how to calculate the standard error for the subordinator parameters $\mu$ and $v$. If the KS weight is used, the covariance matrix for the estimator $\hat\theta_{n,2}$ is approximately equal to (see Definition \ref{defOpt})
\begin{equation}
D_{2,2}^{ - 1}\left[ {\begin{array}{*{20}{c}}
	{ - {D_{2,1}}D_{1,1}^{ - 1}}&I
	\end{array}} \right]\left[ {\begin{array}{*{20}{c}}
	{{S_{1,1}}}&{{S_{1,2}}}\\
	{{S_{2,1}}}&{{S_{2,2}}}
	\end{array}} \right]\left[ {\begin{array}{*{20}{c}}
	{ - {{({D_{1,1}'})}^{ - 1}}{D_{2,1}'}}\\
	I
	\end{array}} \right]{({D_{2,2}'})^{ - 1}}.
\end{equation}
For the SubOU process, the formula for each matrix above except $S_{1,2}$ is shown on page 9. The expression inside each expectation is analytically known. The expectations are computed by Monte Carlo simulation with 200,000 replications. $S_{1,2}=E\left[\tilde{f}_1\left(Y_{t_0}, Y_{t_1}\right)V'({Y_{t_0}},{Y_{t_1}})W({Y_{t_0}})\right]$, where $\tilde{f}_1\left(Y_{t_0}, Y_{t_1}\right)$ is defined in \eqref{eq:f1-tilde}. For the SubOU process, each term inside the expectation for $S_{1,2}$ is known and the expectation is again computed by Monte Carlo simulation with 200,000 replications.

If the optimal weight is used, the covariance matrix for the estimator $\hat\theta_{n,2}$ is approximately equal to (see Definition \ref{defOpt})
\begin{align}
 &{(D_{2,2}^*)^{ - 1}}\left[ {\begin{array}{*{20}{c}}
	{ - D_{2,1}^*D_{1,1}^{ - 1}}&I
	\end{array}} \right]\left[ {\begin{array}{*{20}{c}}
	{{S_{1,1}}}&{S_{1,2}^*}\\
	{S_{2,1}^*}&{S_{2,2}^*}
	\end{array}} \right]\left[ {\begin{array}{*{20}{c}}
	{ - {{({D_{1,1}'})}^{ - 1}}(D_{2,1}^*)'}\\
	I
	\end{array}} \right]{({D_{2,2}^*}')^{ - 1}}\\
&={({D_{2,2}^*}')^{ - 1}}=\left(nE[W^*(Y_{t_0};\theta)'\partial_{\theta_2} V(Y_{t_0},Y_{t_1};\theta)]\right)^{-1}\\
&=\frac{1}{n}\Big(E\left[ {{{\left( {\partial_{\theta_2} V({Y_{t_0}},{Y_{t_1}})} \right)}^\prime }{{\left( {E[V({Y_{t_0}},{Y_{t_1}})V'({Y_{t_0}},{Y_{t_1}})|{Y_{t_0}}]} \right)}^{ - 1}}E[\left( {\partial_{\theta_2} V({Y_{t_0}},{Y_{t_1}})} \right)|{Y_{t_0}}]} \right]\\
&+C_2'Q_1+C_3'Q4 \Big).\quad (\text{by Proposition \ref{prop:optimality} and \ref{prop:optW}})
\end{align}
For the SubOU process, the above inner expectations can be computed in closed-form and the outer expectation is computed by Monte Carlo simulation with 200,000 replications.

We compare using the KS weight and the optimal weight in terms of the standard error for $\mu$ and $v$ when different numbers of eigenfunctions are used. The results are shown in Table \ref{tab:KSandOptW}. We clearly see that the standard errors are very close. Since the KS weight is much easier to compute, we recommend using the KS weight in our method, and it is used in all the following examples and applications.
\begin{table}[htbp!]
	\centering
	\begin{tabular}{rrrr}
		\toprule
		& \shortstack{Number of \\eigenfunctions} & \shortstack{SE for the KS weight} & \shortstack{SE for the optimal weight} \\
		\midrule
		\multicolumn{1}{c}{\multirow{5}[0]{*}{$\mu$}} & \multicolumn{1}{c}{2} & 0.0546609  & 0.0546591 \\
		\multicolumn{1}{c}{} & \multicolumn{1}{c}{3} & 0.0437687 & 0.0437683  \\
		\multicolumn{1}{c}{} & \multicolumn{1}{c}{4} & 0.0434508  & 0.0434503  \\
		\multicolumn{1}{c}{} & \multicolumn{1}{c}{5} & 0.0433678  & 0.0433673 \\
		\multicolumn{1}{c}{} & \multicolumn{1}{c}{6} & 0.0432791 & 0.0432786  \\
		\midrule
		\multicolumn{1}{c}{\multirow{5}[0]{*}{$\nu$}} & \multicolumn{1}{c}{2} & 1.7195609  & 1.7195110  \\
		\multicolumn{1}{c}{} & \multicolumn{1}{c}{3} & 0.3333475 & 0.3333454 \\
		\multicolumn{1}{c}{} & \multicolumn{1}{c}{4} & 0.1813183 & 0.1813178\\
		\multicolumn{1}{c}{} & \multicolumn{1}{c}{5} & 0.1682431 & 0.1682428 \\
		\multicolumn{1}{c}{} & \multicolumn{1}{c}{6} & 0.1531301 & 0.1531295  \\
		\bottomrule
	\end{tabular}%
	\caption{Comparison of the KS weight and the optimal weight}
\label{tab:KSandOptW}%
\end{table}%

\subsection{How Many Eigenfunctions to Use}\label{sec:NumEigen}
We examine the impact of the number of eigenfunctions used in the estimating function on the standard error of the subordinator parameter $\mu$ and $v$. We also calculate the standard error under maximum likelihood estimation which is known to achieve the best statistical efficiency. Theorem 2.2 of \cite{billingsley1961statistical} shows that $\sqrt{n}(\hat{\theta}^{\text{MLE}}_n-\bar{ \theta}) \to \gN(0,S^{-1}_0)$, where
${S_0}= \bar E[\partial_\theta\ln p^\phi(\Delta,Y_0,Y_{\Delta};\theta)\partial_\theta\ln p^\phi(\Delta,Y_0,Y_{\Delta};\theta)']$. To calculate $S_0$, we use the bilinear eigenfunction expansion \eqref{eq:SubDiffTPD} for $p^\phi(\Delta,Y_0,Y_{\Delta})$, which can be computed by truncating the infinite series when the relative accuracy level of $10^{-8}$ is reached. To calculate the expectation for $S_0$, we use Monte Carlo simulation with 200,000 replications.
%The asymptotic covariance matrix for eigenfunction approach is given in section \ref{CANsec}.

Table \ref{tab: chooseEigenNumb} displays the standard error of $\mu,v$ for MLE and various number of eigenfunctions. As expected, MLE gives the smallest standard error, and as the number of eigenfunctions increases, the standard error for both $\mu$ and $v$ decrease to those under MLE. However, the marginal decrease in the standard error is quite small when the number of eigenfunctions is already at $4$. In general, increasing the number of eigenfunctions reduces standard error but requires more computations to obtain the estimator. Using 4 eigenfunctions seems to best balance computational efficiency and statistical efficiency for the IG-SubOU process. Our finding is in line with those in \cite{kessler1999estimating} and \cite{larsen2007diffusion}, who estimate diffusions using eigenfunction-based estimating functions. For the diffusion models considered there, using a few number of eigenfunctions is already good enough.
\begin{table}[htbp!]
	\centering
	\begin{tabular}{rrr}
		\hline
		& \shortstack{SE  for $\mu$}  & \shortstack{SE for $v$} \\
		\hline
		MLE   & 0.0430  & 0.1276  \\
		2 eigenfunctions & 0.0547  & 1.7196 \\
		3 eigenfunctions & 0.0438  & 0.3333  \\
		4 eigenfunctions & 0.0435  & 0.1813  \\
		5 eigenfunctions & 0.0434  & 0.1682  \\
		6 eigenfunctions & 0.0433  & 0.1531  \\
		7 eigenfunctions & 0.0432  & 0.1479  \\
		8 eigenfunctions & 0.0432  & 0.1430  \\
		\hline
	\end{tabular}%
    \caption{Standard error for MLE and various number of eigenfunctions}	\label{tab: chooseEigenNumb}%
\end{table}%

\subsection{The Effect of Numerical Approximation}\label{sec:NumApprox}
In general, the eigenvalues and the eigenfunctions are not known explicitly, and numerical approximations are needed as discussed in Section \ref{approEigenSec}. For the IG-SubOU, since we have explicit expressions for the eigenvalues and the eigenfunctions, we can check the error of using MATSLISE to numerically calculate them. We choose $\Pi_1=\{x_i=-2+0.005i \}_{i=0}^{800}$. The region $[-2,2]$ is large enough for the OU process under the assumed parameter values. The absolute and relative error for the first four nonzero eigenvalues, as well as the maximum absolute error and maximum relative error for the first four nonconstant eigenfunctions on the grid are shown in Table \ref{tab:CPMerror}. In general, as the order of the eigenpair increases, the CPM becomes less accurate, but if the order of the eigenpairs is not too high, we still attain very high level of accuracy.
%\begin{table}[htbp]
%	\centering
%	\begin{tabular}{rrr}
%		\hline
%		Eigenfunction & Max Abs Error &Max Rel Error \\
%		\hline	
%		1&  3.7068e-10& 5.6797e-9  \\
%		2&  6.3483e-10&  4.3054e-9\\
%		3&  3.5525e-9& 4.6452e-9 \\
%		4&   1.3555e-8& 5.3099e-8\\	
%		\hline
%	\end{tabular}%
%    \caption{Error for eigenfunctions}\label{tab:CPMerror}%
%\end{table}%

\begin{table}[htbp!]
	\centering
	\begin{tabular}{crrrr}
		\hline
		Eigenpair & \shortstack{Abs Error\\of Eigenvalue} & \shortstack{Rel Error\\of Eigenvalue} & \shortstack{Max Abs Error \\of Eigenfunction} & \shortstack{Max Rel Error\\of Eigenfunction} \\
		\hline
		1 & 4.1284e-12 & 1.0321e-10 &  3.7068e-10& 5.6797e-9 \\
		2 & 2.1349e-12 & 2.6687e-11 &  6.3483e-10&  4.3054e-9\\
		3 & 3.3295e-12 & 2.7746e-11 &  3.5525e-9& 4.6452e-9 \\
		4 & 3.6774e-12 & 2.2984e-11  &   1.3555e-8& 5.3099e-8\\
		\hline
	\end{tabular}%
	\caption{Error for eigenpairs using the CPM}\label{tab:CPMerror}%
\end{table}%

We also need to numerically calculate the integral \eqref{eq:integral}  to get the KS weight. To do this, we specify two grids, $\Pi_2=\{x_i=-2+0.02i\}_{i=0}^{200}$ and $\Pi_2'=\{x_i=-2+0.01i\}_{i=0}^{400}$. For each grid, we run the Li and Zhang algorithm (\cite{LiZhangSIAM}) to approximate \eqref{eq:integral} and then we extrapolate the results under the two grids to obtain a more accurate approximation.
The estimation results for the two methods are listed in Table \ref{tab: compareTwoAppro}. Here to estimate the standard error, we simulate 100 trajectories with each containing 2000 daily observations to get 100 realizations for the estimator and calculate its sample standard deviation.
From the table, we can see that both methods give almost identical results. Therefore, the numerical approximation proposed in Section \ref{sec:num} works very well.
\begin{table}[htbp!]
	\centering
	\begin{tabular}{cccc}
		\hline
		& Exact Eigenpair  ($\hat{\theta}_n$ (SE)) & Approx Eigenpair ($\hat{\theta}_n^h$ (SE)) \\
		\hline
		$\kappa$ & 0.0426(0.0071)  &0.0426(0.0071) \\
		$\vartheta$ & -0.0027(0.0336) & -0.0027(0.0336) \\
		$\mu$    & 1.0033(0.0429) & 1.0034(0.0429)\\
		$\nu$    & 0.5253(0.1840)  &0.5254(0.1840) \\
		\hline
	\end{tabular}%
	\caption{Estimation results using exact and approximate eigenpairs}\label{tab: compareTwoAppro}%
\end{table}%

%In Theorem \ref{CAN2}, we proved that $\|\hat\theta_n^h-\hat\theta_n\|=O(h_n^2)$ if $p=2$ in \eqref{eq:error_estfunc}. Here we verify the error order for the estimator numerically. We set step size $h_n$ to different levels: $0.02, 0.01, 0.005, 0.004, 0.0033$ and run the Li and Zhang algorithm without extrapolation. In the previous 100 simulated trajectories, we obtain a value for the estimator using the exact estimating function and the approximate estimating function for each given step size. We then, for each step size, take the average of the estimated value across all 100 trajectories. We denote the average value by $\bar \mu_n$, $\bar \nu_n$ for the exact case and by $\bar \mu_n^h$, $\bar \nu_n^h$ for the approximate case. The error $e(h)= \sqrt{(\bar \mu_n^h-\bar \mu_n)^2 +(\bar \nu_n^h-\bar \nu_n)^2}$. We use linear regression to fit the data of $\ln e(h)$ and $\ln h$, which gives us the regression equation $\ln e(h)=2.09 \ln h+7.23$, confirming second order convergence.
%\begin{table}[htbp]
%	\centering
%	\caption{Comparison of Estimators}
%	\begin{tabular}{lrr}
%		\toprule
%		& $|\bar \mu_n^h-\bar \mu_n|$ & $|\bar \nu_n^h-\bar \nu_n|$ \\
%		\midrule
%		h=0.02 & 1.138E-02 & 2.001E-01 \\
%		h=0.01 & 4.526E-04 & 1.684E-01 \\
%		h=0.005 & 4.583E-04 & 7.188E-02 \\
%		h=0.004 & 4.180E-04 & 1.758E-02 \\
%		h=0.0033 & 5.672E-04 & 2.233E-03 \\
%		\bottomrule
%	\end{tabular}%
%	\label{tab:compEstimator}%
%\end{table}%

\subsection{The Impact of Data Frequency}\label{sec:DataFreq}
We examine the impact of data frequency on the standard error of the estimator. Now we set the sampling time interval $\Delta=0.01$ days but keep the same sampling period. Thus, there are 200,000 observations in the sampling period of 2000 days. The standard error of the estimator, which is again estimated from 100 simulated trajectories, is shown in Table \ref{tab: highFre}. Increasing data frequency in the same sampling period has little impact on estimating $\vartheta$, the long-run mean of the OU process. To reduce its standard error, the sampling period should be increased. Higher frequency of data does reduce the standard error of $\kappa$ (the mean-reversion speed of the OU process), $\mu$ (mean rate of the subordinator) and $v$ (variance rate of the subordinator). In particular, by sampling 100 times faster, we achieve a reduction by around 50\% in the standard error of $v$. For $\kappa$ and $\mu$, the reduction ratio is much smaller but still significant.
\begin{center}
	\begin{table}[htbp!]
		\centering
		\begin{tabular}{ccccc}
			\hline
			& $\kappa$ & $\vartheta$ & $\mu$ & $\nu$ \\
			\hline
			Estimate & 0.0423 & -0.0003 & 1.0020 & 0.5063\\
			SE & 0.0065 & 0.0336 & 0.0283 & 0.0928\\
			\hline
		\end{tabular}%
\caption{Estimation results using high frequency data}
		\label{tab: highFre}%
	\end{table}%
\end{center}

\subsection{A Subordinate Diffusion Model for Fitting VIX Data}\label{sec:VIX}
The CBOE volatility index, known as VIX, is a well-known fear gauge with large volume of futures and options contracts written on it. Using high-frequency data of VIX and applying non-parametric statistical tools developed in \cite{todorov2010activity}, \cite{TodorovTauchen} shows that VIX follows a pure jump process with infinite jump activity and infinite jump variation. Here, we develop a parametric pure jump model with these features based on subordinate diffusions for fitting VIX data.

\cite{goard2013stochastic} analyzed the fit of a class of diffusion models to VIX and concluded that the 3/2 diffusion, which is the solution to the SDE $dX_t=\kappa X_t(\vartheta-X_t)dt+\sigma X^{3/2}_tdW_t$, achieves the best fit. Here, we consider a more general class of diffusions than \cite{goard2013stochastic}. We assume that $X$ follows
\begin{equation}\label{eq:VIXDiff}
dX_t=\mu(X_t)dt+\sigma X_t^\beta dW_t,\  \mu(x)=\sum_{i=0}^{k}c_i x^i,\ k\ge 1,\ \beta>0.
\end{equation}
The volatility of VIX is positively correlated with VIX itself, so it is reasonable to assume $\beta>0$. We then construct a subordinate diffusion by time changing $X$ with an independent inverse Gaussian subordinator without drift. We choose this subordinator because it can be shown that using it makes the time changed process $Y$ a pure jump process with infinite jump activity and infinite jump variation, capturing the features found in \cite{TodorovTauchen} (the proof is similar to the proof of Proposition 2 in \cite{LLZVIX}). Due to the scaling invariance result in Corollary \ref{cor:TSS}, we fix $\sigma=1$ and estimate the other parameters.

For $X$ to satisfy Assumption 1, we need to impose some restrictions on the parameters. The sufficient condition for $X$ to have purely discrete spectrum can be derived by applying Theorem 3.3 of \cite{linetsky2007spectral}, which requires $\beta \ne 1$ or $\beta=1, k\ge 2$. The sufficient condition that guarantees ergodicity for diffusions of form \eqref{eq:VIXDiff} can be found in \cite{conley1997short}. Together, we have the following restrictions: $\beta=1$, $k\ge 2$, $c_0>0$, $c_k<0$ or
$\beta\ge 1/2$, $\beta\ne 1$, if $2\beta<k+1$, $c_0>0$, $c_k<0$; if $2\beta=k+1$, $c_0>0$, $c_k<\frac{1}{2}$; if $2\beta>k+1$, $c_0>0$ and no restrictions on $c_k$; if $\beta=1/2$, $c_0>0$ is replaced by $c_0>\frac{1}{2}$ in the previous statements. In our estimation, we set $k=2$.

To estimate the diffusion parameters, we first apply \eqref{eq: diffCon1} with $f(x)=\partial_{c}\log q(x)$ (this denotes the column vector of partial derivative w.r.t. $c_0,c_1,c_2$) and obtain the following moment condition
\[E\left[ \mu(Y_t)d'(Y_t) + \frac{1}{2}{\sigma ^2}Y_t^{2\beta }d''(Y_t) \right] = 0,\]
where
\[d'(Y_t) = \frac{\partial^2}{\partial y\partial c}\log q(Y_t)=\left[ {\begin{array}{*{20}{c}}
	{2Y_t^{ - 2\beta }}\\
	{2Y_t^{1 - 2\beta }}\\
	{2Y_t^{2 - 2\beta }}
	\end{array}} \right],d''(Y_t) = \left[ {\begin{array}{*{20}{c}}
	{ - 4\beta Y_t^{ - 1 - 2\beta }}\\
	{2(1 - 2\beta )Y_t^{ - 2\beta }}\\
	{2(2 - 2\beta )Y_t^{1 - 2\beta }}
	\end{array}}\right] .\]
We add another moment condition to estimate $\beta$, which is
\begin{equation}
E\left[\mu(Y_{t_1})\left( {g({Y_{t_1} } - {Y_{t_0}}) + g({Y_{t_0}} - {Y_{t_1} })} \right)\right] +E\left[ \frac{1}{2}{\sigma ^2}Y_{t_1}^{2\beta }\left( {g'({Y_{t_1} } - {Y_{t_0}}) - g'({Y_{t_0}} - {Y_{t_1} })} \right) \right] = 0,
\end{equation}
with $g(x)=\frac{1}{\delta \sqrt{2\pi}}e^{-\frac{ (x-m)^2}{2\delta ^2}}$. Following \cite{conley1997short}, we set $m=0$ and $\delta$ to be the $50\%$ quantile of the empirical distribution of $\{ |Y_{t_i}-Y_{t_{i-1} }|\}$, which is 0.0056 in our case. To estimate $\mu$ and $v$ which are the mean rate and variance rate of the inverse Gaussian subordinator, we use the estimating function \eqref{eq: F2-1} with $M=6$. The eigenvalues and eigenfunctions are first numerically calculated by the CPM, and then we run Li and Zhang algorithm with extrapolation to calculate \eqref{eq:integral} to obtain the KS weight.

We also estimate the diffusion model \eqref{eq:VIXDiff} as a benchmark. It is not difficult to see that the general diffusion model can be expressed as the diffusion in \eqref{eq:VIXDiff} with $\sigma=1$ time changed with a deterministic clock $T_t=\gamma t$. This model is again estimated using the two-step procedure above by first estimating the diffusion parameters $(c_0,c_1,c_2,\beta)$ and then $\gamma$.

The data we use is daily VIX data in the period 03/01/2012 to 31/07/2014 with a total of 900 days downloaded from CBOE. Figure \ref{VIX} plots the sample path of VIX, which is strongly mean-reverting with large moves. The estimation results for the diffusion model and the subordinate diffusion model are listed in Table \ref{tab:estResult}. Here, the time unit is year and $\Delta=\frac{1}{252}$ years.

% Table generated by Excel2LaTeX from sheet 'para'
\begin{table}[htbp]
	\centering
	\begin{tabular}{crrrrrrr}
		\toprule
		para ($\sigma=1$)   & $c_0$    & $c_1$    & $c_2$    & $\beta$  & $\mu$    & $\nu$    & $\gamma$ \\
		\midrule
		\shortstack{diffusion} &  \shortstack{0.0053} & \shortstack{0.0177 } & \shortstack{-0.3273} & \shortstack{2.3695}     & N.A.    & N.A.    & 225.3732 \\
		\shortstack{jump  } &  \shortstack{0.0053} & \shortstack{0.0177 } & \shortstack{-0.3273} & \shortstack{ 2.3695}  &  \shortstack{222.2240}  & \shortstack{212.2837}  & N.A. \\
		\bottomrule
	\end{tabular}%
	\caption{Estimated parameters of the diffusion and the subordinate diffusion model}\label{tab:estResult}%
\end{table}%

We employ the hypothesis test used in \cite{larsen2007diffusion} to test the goodness of fit. This test evaluates whether $U_i:=F(Y_{t_i}|Y_{t_{i-1}};\theta)$ are i.i.d. uniform random variables over $[0,1]$ ($F(y|x;\theta)$ is the conditional distribution function for the parametric model). To do this, we first calculate $U_i$s for the model under consideration using the estimated parameter values. Then, we apply the Kolmogorov-Smirnov test and the $\chi^2$ test. The test statistics and $p$-value are shown in Table \ref{tab:compareTwoModel}, which shows that the diffusion model is rejected at all practical significance levels while the subordinate diffusion model fits the data well. The superiority of the subordinate diffusion model over the diffusion counterpart is also evidenced from the QQ plot in Figure \ref{qqplot}.
\begin{table}[htbp!]
	\centering
	\begin{tabular}{rrrrrr}
		\toprule
		&       &       & K-S test  & $\chi^2$ test (20 bins) & $\chi^2$ test (100 bins) \\
		\midrule
		\multicolumn{1}{c}{\multirow{2}[4]{*}{\shortstack{SubDiff(Jump)}}} & \multicolumn{1}{c}{} & statistic & 0.0374 & 20.0444  & 88.8889 \\
		\multicolumn{1}{c}{} & \multicolumn{1}{c}{} & p-value & 15.63\% & 12.88\% & 62.97\% \\
		\midrule
		\multicolumn{1}{c}{\multirow{2}[4]{*}{\shortstack{Diffusion }}} & \multicolumn{1}{c}{} & statistic & 0.0599 & 45.7333 & 150.0000 \\
		\multicolumn{1}{c}{} & \multicolumn{1}{c}{} & p-value & 0.3\% & 0.0059\% & 0.0277\% \\
		\bottomrule
	\end{tabular}%
	\caption{Goodness of fit test results}\label{tab:compareTwoModel}%
\end{table}%

\begin{figure}[ht]
	\centering
	\subfigure[VIX historical data]{
		\includegraphics[width = 76mm,height=60mm]{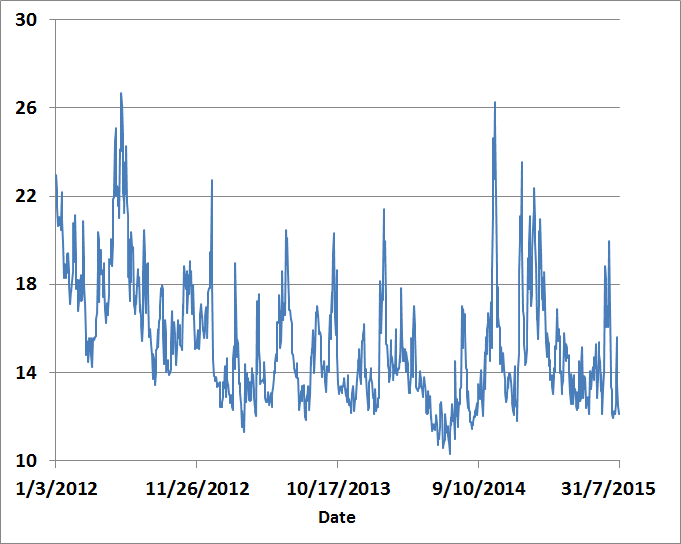}
		\label{VIX}
	}
	\subfigure[QQ plot]{
		\includegraphics[width = 76mm,height=60mm]{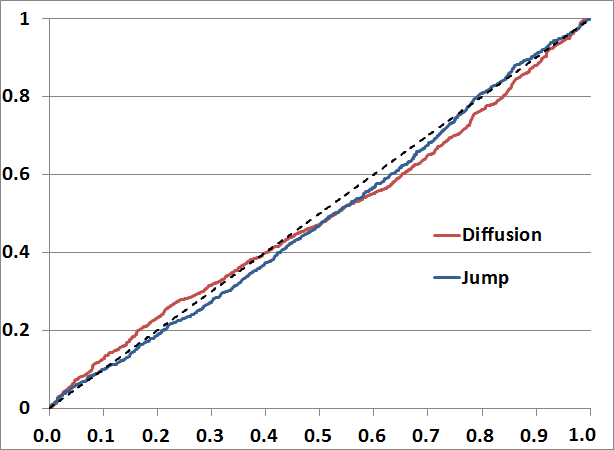}
		\label{qqplot}
	}
	\caption{VIX historical data and QQ Plot}\label{fig:VIX}
\end{figure}

\section{Conclusions}\label{sec:conclusions}
This paper considers parametric inference for a discretely observed ergodic subordinate diffusion. In general, we can only identify the characteristics of a subordinate diffusion up to scale and a two-step estimation procedure based on estimating functions is proposed, which is computationally and statistically efficient. The estimating function in the first step is based on the moment conditions that only involve diffusion parameters. In the second step, a martingale estimating function constructed using eigenvalues and eigenfunctions of the subordinate diffusion is used to estimate the parameters of the L\'{e}vy subordinator. Our method does not require the eigenpairs to be known analytically. For the general case, we develop an efficient numerical procedure to calculate the martingale estimating function. Under regularity conditions, consistency and asymptotic normality of our estimator are established considering the effect of numerical approximation. In future research, we will apply our method to estimate subordinate diffusion models in other applications, such as fitting commodity price data. We also plan to develop efficient methods to estimate other types of time-changed diffusions, for which the time change is not a L\'{e}vy subordinator but a more complicated process (see \cite{LiLinetsky} and \cite{MendozaCarrLinetsky}). These models exhibit stochastic volatility and are very useful in financial applications.

%we plan to develop asymptotic analysis for subordinate diffusion sampled at high-frequency. As the example in Section \ref{sec:DataFreq} hints, for some %diffusion parameters and subordinator parameters, convergence is also observed when the data frequency increases to infinity while the sampling period is %fixed. This is different from the current setting in which the data frequency is fixed while the sampling period approaches infinity.

\appendix
\section{Proofs}

\noindent{\textbf{Lemma \ref{lemma:ergodicity}}}:
From Proposition 6 in \cite{hansen1995back}, $\{Y_t, t\ge0\}$ is ergodic if and only if $\gG^\phi f=0$ for $f\in \gD(\gG^\phi)\bigcap\gL(q)$ implies that $f=0$. Here $\gD(\gG^\phi)$ is the domain of $\gG^\phi$ and $\gL (q)=\{f\in L^2(I,q):\int_{l}^{r}f(x)q(x)dx=0\}$. We verify this equivalent condition. For $f\in \gD(\gG^\phi)\bigcap\gL(q)$ such that $\gG^\phi f=0$, suppose that $f$ is not identically zero. Since $\gD(\gG^\phi)\subseteq\gD(\gG)$, $f\in \gD(\gG)\bigcap\gL(q)$.  Then $\gG^\phi f=0$ shows that $f$ is an eigenfunction of $\gG^\phi$ with eigenvalue equal to $0$. Since $\gG$ and $\gG^\phi$ have the same set of eigenfunctions, we must also have $\gG f=\lambda f$ for some $\lambda\le 0$. However, $\gG^\phi f=-\phi(-\lambda)f$ by \eqref{eq:SubDiffGenEigen}. Therefore $\phi(-\lambda)=0$, which further implies $\lambda=0$ by \eqref{eq:LTLevySub}, and hence $\gG f=0$ for some nonzero $f$. This contradicts that $\{X_t,t\ge0\}$ is ergodic (which we assume in Assumption \ref{assump:diff} (1)), because applying \cite{hansen1995back}, Proposition 6 again shows that when $\{X_t, t\ge0\}$ is ergodic, $\gG f=0$ for $f\in \gD(\gG)\bigcap\gL(q)$ implies that $f=0$. So now we can conclude that $\gG^\phi f=0$ for $f\in \gD(\gG^\phi)\bigcap\gL(q)$ implies that $f=0$ and $\{Y_t, t\ge0\}$ is ergodic.

To show that $\{Y_t:t=0,\Delta,2\Delta,\cdots \}$ is ergodic, we can verify the equivalent condition in Proposition 7 of \cite{hansen1995back}. The arguments are similar to the above and are omitted here to save space.

To show that $q(x)$ is the stationary density for $Y$, let $p(t,x,y)$ be the transition density of $X$ and $s_t(du)$ be the distribution of $T_t$. Then $Y$'s transition density $p^\phi(t,x,y)=\int_{(0,\infty)}p(u,x,y)s_t(du)$. We have
\begin{align}
\int_I p^\phi(t,x,y)q(x)dx &= \int_I \int_{(0,\infty)}p(u,x,y)s_t(du) q(x)dx=\int_{(0,\infty)}\int_I p(u,x,y)q(x)dx s_t(du)\\
&=\int_{(0,\infty)}q(y)s_t(du)=q(y)\int_{(0,\infty)}s_t(du)=q(y).
\end{align}
This shows the claim. \qed

\bigskip
\noindent{\textbf{Theorem \ref{thm:idenPPS}}}: ``$\Rightarrow$'': We first prove the implications of $p^\phi_1(t,x,y)=p^\phi_2(t,x,y)$. Let $P_i^\mu$ be the law of
$Y$ with transition density $p^\phi_i(t,x,y)$ and initial distribution $\mu$ ($i=1,2$). Then $p^\phi_1(t,x,y)=p^\phi_2(t,x,y)$ implies $P_1^\mu=P_2^\mu$ for any $\mu$. Since $\{Y_t,t\ge0\}$ is ergodic by Lemma \ref{lemma:ergodicity}, we have for any initial distribution $\mu$ and any $f$ such that $\int_l^r f(x)q_i(x)dx$ ($i=1,2$) is finite (note that $q_i(x)$ is the stationary density for $Y$ under $p^\phi_i(t,x,y)$ by Lemma \ref{lemma:ergodicity}),
\begin{equation}
\frac{1}{t}\int_0^t {f({Y_t})dt}  \to \int_l^r f(x)q_i(x)dx,\ P^\mu_i-a.s.,\ i=1,2.
\end{equation}
Since $P_1^\mu=P_2^\mu$, we must have $\int_l^r f(x)q_1(x)dx=\int_l^r f(x)q_2(x)dx$. Taking $f(x)=1_{\{x\le y\}}$ for arbitrary $y\in(l,r)$ shows $q_1(x)=q_2(x)$ for $x\in (l,r)$. If a boundary point is finite and included in the state space, using the continuity of $q_i(x)$ shows that $q_1(x)=q_2(x)$ at the boundary point. To sum up, $q_1(x)=q_2(x)$ for $x\in I$. In the following, we will denote the common stationary density by $q(x)$ and all common quantities will be denoted without subscript $i$. The rest of the proof consists of two steps.

\noindent Step 1: We will compare $\theta_1(x)$ and $\theta_2(x)$. Two cases are considered below.

\noindent Case 1. Suppose $X$ lives on a compact interval $[l,r]$ with two reflecting boundaries. In this case, the spectrum of $X$ is discrete and so is the spectrum of $Y$. Note that $p^\phi_1(t,x,y)=p^\phi_2(t,x,y)$ implies that $\gP^\phi_{1t}=\gP^\phi_{2t}$ for any $t>0$, so they share the same orthonormal set of  eigenfunctions, and so do $\gG_1$ and $\gG_2$ (because a function is an eigenfunction of $\gG_i$ if and only if it is an eigenfunction of $\gP^\phi_{it}$).
Take one eigenfunction $\varphi(x)$. We have
\begin{equation}
{\gG_1}\varphi(x) = {\lambda _1}\varphi(x),\ 	{\gG_2}\varphi(x) = {\lambda _2}\varphi(x).
\end{equation}
Note that $\gG_i \varphi(x)=\mu_i(x)\varphi'(x)+\frac{1}{2} \sigma^2_i(x)\varphi''(x)$. Multiplying both sides of the equation by $q(x)$ and noting that $\mu_i(x) q(x)= \frac{1}{2}(\sigma^2_i(x)q(x))'$, we get
	\[\frac{1}{2}({\sigma_1^2}{(x)}q(x))'\varphi'(x) + \frac{1}{2}({\sigma_1^2}{(x)}q(x))\varphi'' (x) = {\lambda_1}\varphi (x)q(x)\]
	or
	\[\frac{1}{2}({\sigma_1^2}{(x)}q(x)\varphi'(x))' = {\lambda_1}\varphi(x)q(x).\]
	Integrating on both sides from $l$ to $y$ and using $\varphi'(l)=0$, we obtain
	\begin{equation} \label{eq: 8.1}
		{\sigma_1^2}{(y)}q(y)\varphi' (y) = 2{\lambda _1}\int_l^y {\varphi (x)q(x)} dx .
	\end{equation}
	Similarly, we have
	\begin{equation} \label{eq: 8.2}
		{\sigma_2^2}{(y)}q(y)\varphi' (y) = 2{\lambda _2}\int_l^y {\varphi (x)q(x)} dx.
	\end{equation}
	Combining \eqref{eq: 8.1} and \eqref{eq: 8.2}, it is clear that $\sigma_1(\cdot)$ and $\sigma_2(\cdot)$ are proportional on $\{y\in I:\int_l^y {\varphi (x)q(x)} dx\ne 0\}$. Since $\{x:\varphi(x)=0\}$ has a finite number of points in light of Theorem 4.1 of \cite{hansen1998spectral} and $q(x)>0$ for all $x$,  there are finite number of $y$ such that $\int_l^y \varphi(x)q(x) dx= 0$. Hence, except on a finite number of points, $\sigma_1(x)$ and $\sigma_2(x)$ are proportional. However, from the continuity of $\sigma_1(x)$ and $\sigma_2(x)$ (Assumption \ref{assump:diff} (3)), they must be proportional on the entire $I$, i.e., there exists $c>0$ such that $\sigma_2(x)=\sqrt{c}\sigma_1(x)$ for all $x\in I$. Because $\mu_i(x) q(x)= \frac{1}{2}(\sigma^2_i(x)q(x))'$,
we have $\mu_2(x)=c\mu_1(x)$ for all $x\in I$. Together, we have $\theta_2(x)=c\theta_1(x)$ for all $x\in I$.
	
\noindent Case 2. We now consider $X$ with general state space $I$. For any finite $[a,b] \subset I$, we consider a diffusion $X'$ living on $[a,b]$ with the same drift and diffusion coefficient as $X$ on this interval, and $X'$ is reflected at $a,b$. Denote the generators of $X'$ associated with $\theta_1(x)$ and $\theta_2(x)$ by $\gA_1$ and $\gA_2$. Note that the stationary density of $X'$ is proportional to the stationary density of $X$ on $[a,b]$, so it is the same under $\theta_1(x)$ and $\theta_2(x)$. We will show that $\gA_1$ and $\gA_2$ have the same set of eigenfunctions.

Suppose $\varphi(x)$ is an eigenfunction of $\gA_1$ with corresponding eigenvalue $\lambda_1\le 0$. Then
\[\gA_1 \varphi(x)=\lambda_1 \varphi(x),\qquad \varphi '(a)=\varphi'(b)=0. \]
From the discussions on p.788 of \cite{hansen1995back}, $\varphi(x)$ has derivatives up to the fourth order on $(a,b)$ and all of them have finite limits when $x$ approaches $a$ and $b$. Find $\delta>0$ small enough such that $[a-\delta,b+\delta]\subset I$. Define a new function $\hat{\varphi}(x)$ on $I$ such that $\hat{\varphi}(x)=\varphi(x)$ on $(a,b)$, $\hat{\varphi}(x)=0$ for $x \ge b+\delta$ and $x\le a-\delta$, $\lim_{x \to a} {{\hat \varphi }^{(n)}}(x) = \lim_{x \to a} {\varphi ^{(n)}}(x)$, $\lim_{x \to b} {{\hat \varphi }^{(n)}}(x) = \lim_{x \to b} {\varphi ^{(n)}}(x)$ for $0 \le n\le 4$, and $\hat{\varphi}(x)$ has derivatives up to the fourth order on $(a-\delta,b+\delta)$. Note that $\hat{\varphi}(x)$ has compact support and it belongs to $D(\gG_1)$ and $D(\gG_2)$.
Moreover, $\gG_1 \hat{\varphi}(x)$ and $\gG_2 \hat{\varphi}(x)$ are twice continuously differentiable and with compact support, thus $\gG_1 \hat{\varphi}(x)$ $(\gG_2 \hat{\varphi}(x))$ is in $D(\gG_2)$ $ (D(\gG_1))$. It is easy to see that $\gG_1\gG_2 \hat{\varphi}(x) =\gG_2\gG_1 \hat{\varphi}(x)$. Thus, we have
\[\gA_1\gA_2 \varphi (x)=\gG_1\gG_2 \hat{\varphi}(x) =\gG_2\gG_1 \hat{\varphi}(x)=\gA_2\gA_1\varphi(x)=\lambda_1 \gA_2 \varphi(x),\qquad  x\in (a,b).\]
This shows that $\gA_2 \varphi(x)$ is an eigenfunction of $\gA_1$ with eigenvalue $\lambda_1$, which further implies that $\gA_2\varphi(x) =\lambda_2 \varphi(x)$ for some constant $\lambda_2$, and hence $\varphi(x)$ is also an eigenfunction of $\gA_2$. Conversely, using similar arguments, one can show that if $\varphi(x)$ is an eigenfunction of $\gA_2$, then it is also an eigenfunction of $\gA_1$. This shows that $\gA_1$ and $\gA_2$ have the same set of eigenfunctions.

Now we can repeat the procedure used in Case 1 and conclude that for any finite $[a,b]\subset I$, there exists $c>$ such that $\theta_2(x)=c\theta_1(x)$ for all $x\in [a,b]$. Due to the continuity of $\theta_1(x)$ and $\theta_2(x)$, $c$ does not depend on $[a,b]$ and since $[a,b]$ is arbitrary, $\theta_2(x)=c\theta_1(x)$ holds for all $x\in I$.

\noindent Step 2: We now show the implications on the characteristics of the subordinator. We have proved that $\theta_2(x)=c\theta_1(x)$ for $x\in I$ for some constant $c>0$. Then $\gG_2 f(x)=c\gG_1f(x)$ for $f \in \gD(\gG_1)\cap \gD(\gG_2)$. We have already explained in the proof of Case 1 that $p^\phi_1(t,x,y)$ $p^\phi_2(t,x,y)$ share the same orthonormal set of eigenfunctions which we denote by $\{\lambda_n, n=0,1,\cdots\}$. Let $\lambda_n$ be the eigenvalue of $\gG_1$ for $\varphi_n(x)$, then $c\lambda_n$ is the eigenvalue of $\gG_2$ for $\varphi_n(x)$. For each $\varphi_n(x)$, we have
\[\gP_t^{{\phi _1}} \varphi _n= {e^{ - \phi _1 ( - {\lambda _n})t}}{\varphi _n},\ \gP_t^{{\phi _2}} \varphi _n= {e^{ - \phi _2 ( - {c\lambda _n})t}}{\varphi _n}.\]
$\gP _t^{\phi _1} =\gP _t^{\phi _2}$ implies that
\begin{equation} \label{eq: relationOfLaplace}
	{\phi _1}( - \lambda_n ) = {\phi _2}( - c\lambda_n )\quad \text{for}\ n=0,1,2,\cdots.
\end{equation}
From \cite{bertoin1999subordinators}, p.7
	\[\mathop {\lim }\limits_{n \to \infty } \frac{{{\phi _1}( - {\lambda _n})}}{{ - {\lambda _n}}} = \mathop {\lim }\limits_{\lambda  \to \infty } \frac{{{\phi _1}(\lambda )}}{\lambda } = \mathop {\lim }\limits_{\lambda  \to \infty } \left( {{\gamma _1} + \int_0^\infty  {\frac{{1 - {e^{ - \lambda s}}}}{\lambda }{\nu _1}(ds)} } \right) = {\gamma _1}\]
	and
	\[\mathop {\lim }\limits_{n \to \infty } \frac{{{\phi _2}( - c{\lambda _n})}}{{ - c{\lambda _n}}} = \mathop {\lim }\limits_{\lambda  \to \infty } \frac{{{\phi _2}(\lambda )}}{\lambda } = \mathop {\lim }\limits_{\lambda  \to \infty } \left( {{\gamma _2} + \int_0^\infty  {\frac{{1 - {e^{ - \lambda s}}}}{\lambda }{\nu _2}(ds)} } \right) = {\gamma _2}.\]
	By \eqref{eq: relationOfLaplace}, we have
	\[\mathop {\lim }\limits_{n \to \infty } \frac{{{\phi _2}( - c{\lambda _n})}}{{ - c{\lambda _n}}} = \mathop {\lim }\limits_{n \to \infty } \frac{{{\phi _1}( - {\lambda _n})}}{{ - c{\lambda _n}}} = \frac{{{\gamma _1}}}{c}.\]
	Thus $\gamma_1=c\gamma_2.$
	From \cite{bertoin1999subordinators}, Section 1.2, we have
\begin{equation}\label{eq:phi-rel}
\frac{{{\phi _1}( - {\lambda _n})}}{{ - {\lambda _n}}} - {\gamma _1} = {{\hat \omega }_1}( - {\lambda _n}),\quad \frac{{{\phi _2}( - c{\lambda _n})}}{{ - c{\lambda _n}}} - {\gamma _2} = {{\hat \omega }_2}( - c{\lambda _n})\quad \text{for}\ n=0,1,2,\cdots.
\end{equation}
Since
$\frac{{{\phi _1}( - {\lambda _n})}}{{ - {\lambda _n}}} - {\gamma _1} = c\left( {\frac{{{\phi _2}( - c{\lambda _n})}}{{ - c{\lambda _n}}} - {\gamma _2}} \right)$, we have
	${{\hat \omega }_1}( - {\lambda _n}) = c{{\hat \omega }_2}( - c{\lambda _n})$ for all $n$.

\smallskip
``$\Leftarrow$'' We now prove that the conditions in \eqref{eq:idenPPSConditions} implies $p^\phi_1(t,x,y)=p^\phi_2(t,x,y)$. From \eqref{eq:phi-rel}, we see that \eqref{eq:idenPPSConditions} gives us that $\phi_1(-\lambda_n)=\phi_2(-c\lambda_n)$. Our assumption implies that both $p^\phi_1(t,x,y)$ and $p^\phi_2(t,x,y)$ are given by the bilinear eigenfunction expansion \eqref{eq:SubDiffTPD}. Using the given conditions and the expansion, it is straightforward
verify that $p^\phi_1(t,x,y)=p^\phi_2(t,x,y)$.
\qed

\bigskip
\noindent{\textbf{Proof of Corollary \ref{cor:TSS}}}: For tempered stable subordinators, using \eqref{eq:TSLevyMeasure}, we obtain that
\begin{equation}
\hat \omega (\lambda ) = - C\Gamma ( - p)[{{(\lambda  + \eta )}^p} - {\eta ^p}]/\lambda,\ \lambda>0.
\end{equation}
The condition $\hat \omega_1(\lambda_n)=c\hat{\omega}_2(c\lambda_n)$ for all eigenvalues $\lambda_n>0$ becomes that the equation
\begin{equation}
- C_1\Gamma ( - p_1)[{{(\lambda_n  + \eta_1 )}^{p_1}} - {\eta_1^{p_1}}]/\lambda_n=- C_2\Gamma(-p_2)[{{(c\lambda_n  + \eta_2 )}^{p_2}} - {\eta_2^{p_2}}]/\lambda_n
\end{equation}
has infinitely many solutions on $(0,\infty)$. It is not difficult to show that this is equivalent to $p_1=p_2$, $\eta_1=\frac{1}{c} \eta_2$ and $C_1=c^{p_1}C_2$.\qed

\bigskip
\noindent{\textbf{Proof of Proposition \ref{prop:optimality}}}:
Denote
\begin{equation} \label{eq: defY}
	Y = \left[ {\begin{array}{*{20}{c}}
			{ - {D_{2,1}}D_{1,1}^{ - 1}}&I
		\end{array}} \right]\left[ {\begin{array}{*{20}{c}}
		{{S_{1,1}}}&{{S_{1,2}}}\\
		{{S_{2,1}}}&{{S_{2,2}}}
	\end{array}} \right]\left[ {\begin{array}{*{20}{c}}
	{ - {{(D_{1,1}')}^{ - 1}}D_{2,1}'}\\
	I
\end{array}} \right].
\end{equation}
$Y^*$ is defined similarly with $Y$ by replacing $D_{2,1}, S_{1,2}, S_{2,1}, S_{2,2}$ with $D_{2,1}^*, S_{1,2}^*, S_{2,1}^*, S_{2,2}^*$. Since $W^*$ is optimal, \begin{equation} \label{eq: defM}
M:=D_{2,2}^{ - 1}Y{(D_{2,2}')^{ - 1}} - {(D_{2,2}^*)^{ - 1}}{Y^*}{({D_{2,2}^*}')^{ - 1}}
\end{equation}
is positive semi-definite for all possible $W$. Note that
\[Y - {D_{2,2}}{(D_{2,2}^*)^{ - 1}}{Y^*}{(D{_{2,2}^* }')^{ - 1}}D_{2,2}' = {D_{2,2}}M D_{2,2}'\]
is positive semi-definite. Now let $\tilde{F}_{n,2}=F^*_{n,2}+\alpha F_{n,2}$ and $\tilde{Y}$ is still defined in the same way as \eqref{eq: defY} by replacing $D$ and $S$ of $F_{n,2}$ with those of $\tilde{F}_{n,2}$. Then
	\begin{align}
		&\tilde{Y} - {\tilde{D}_{2,2}}{{(D_{2,2}^*)}^{ - 1}}{Y^*}{{({D_{2,2}^*}')}^{ - 1}}{\tilde{D}_{2,2}}'\notag \\
		&={\alpha ^2}\left(Y - {D_{2,2}}{{(D_{2,2}^*)}^{ - 1}}{Y^*}{{({D_{2,2}^*}')}^{ - 1}} D_{2,2}' \right)\notag \\
		&+\alpha \left( \left[ {\begin{array}{*{20}{c}}
				{ - D_{2,1}^*D_{1,1}^{ - 1}}&I
			\end{array}} \right]\left[ {\begin{array}{*{20}{c}}
			{{S_{1,1}}}&{{S_{1,2}}}\\
			{S_{2,1}^*}&{\hat{S}_{2,2}}
		\end{array}} \right]\left[ {\begin{array}{*{20}{c}}
		{ - {{(D_{1,1}^\prime )}^{ - 1}}(D_{2,1})'}\\
		I
	\end{array}} \right]-{Y^*}{{({D_{2,2}^*}')}^{ - 1}}D_{2,2}' \right) \\
	&+\alpha \left( \left[ {\begin{array}{*{20}{c}}
			{ - D_{2,1}D_{1,1}^{ - 1}}&I
		\end{array}} \right]\left[ {\begin{array}{*{20}{c}}
		{{S_{1,1}}}&{S_{1,2}^*}\\
		{{S_{2,1}}}&{\hat{S}_{2,2}'}
	\end{array}} \right]\left[ {\begin{array}{*{20}{c}}
	{ - {{(D_{1,1}^\prime )}^{ - 1}}(D_{2,1}^*)'}\\
	I
\end{array}} \right]- {D_{2,2}}{{(D_{2,2}^*)}^{ - 1}}{Y^*}  \right). \label{eq: Ypart1}
\end{align}
This is of form $\alpha^2H_1+\alpha H_2$, where $H_1$ is positive semi-definite since $W^*$ is optimal. $\alpha^2H_1+\alpha H_2$ is positive semi-definite only when $H_2=0$.
The equation that $H_2=0$ can be rewritten as $Z'+Z=0$ where $Z$ is the expression in the parenthesis of \eqref{eq: Ypart1}.
Now, replace $F_{n,2}$ by $EF_{n,2}$ where $E=diag(e_1,\cdots,e_n)$ is an arbitrary diagonal matrix with constant diagonal entries. Then, we have $Z'E+EZ=0$ or
$e_jZ_{ji}+e_iZ_{ij}=0$ for every $i,j$. Since $E$ is arbitrary, we must have $Z=0$, and hence
\begin{equation} \label{eq: getOptCon}
	D_{2,2}^{ - 1}\left[ {\begin{array}{*{20}{c}}
			{ - {D_{2,1}}D_{1,1}^{ - 1}}&I
		\end{array}} \right]\left[ {\begin{array}{*{20}{c}}
		{{S_{1,1}}}&{S_{1,2}^*}\\
		{{S_{2,1}}}&{\hat{S}'_{2,2} }
	\end{array}} \right]\left[ {\begin{array}{*{20}{c}}
	{ - {{({D_{1,1}}')}^{ - 1}}(D_{2,1}^*)'}\\
	I
\end{array}} \right]= (D_{2,2}^*)^{-1}{Y^*}
\end{equation}
for every possible $W$. This shows the only if part of the statement.

Now suppose that \eqref{eq: getOptCon} holds. We want to prove $M$, which is given in \eqref{eq: defM}, is positive semi-definite for all $W \in \mathcal{W}$. Denote by
\[G = \left[ {\begin{array}{*{20}{c}}
	{ - {D_{2,1}}D_{1,1}^{ - 1}}&I
	\end{array}} \right]\left[ {\begin{array}{*{20}{c}}
	{{F_{n,1}}}\\
	{{F_{n,2}}}
	\end{array}} \right],\quad {G^*} = \left[ {\begin{array}{*{20}{c}}
	{ - {D_{2,1}^*}D_{1,1}^{ - 1}}&I
	\end{array}} \right]\left[ {\begin{array}{*{20}{c}}
	{{F_{n,1}}}\\
	{{F_{n,2}^*}}
	\end{array}} \right].\]
Using the definition of $Y$ and $Y^*$ (\eqref{eq: defY}), one can show that $Y=E[GG']$, $Y^*=E[G^*(G^*)']$, and \eqref{eq: getOptCon} can be reformulated as $D_{2,2}^{-1}E[G(G^*)']=(D_{2,2}^*)^{-1}{Y^*}$.
Now we have
\begin{align}
	& E\left[ {\left( {D_{2,2}^{ - 1}G - {{(D_{2,2}^*)}^{ - 1}}{G^*}} \right)\left( {G'{{(D_{2,2}')}^{ - 1}} - ({G^*})'{{(D{{_{2,2}^*}^\prime })}^{ - 1}}} \right)} \right] \\
	&=D_{2,2}^{ - 1}Y{(D_{2,2}^\prime )^{ - 1}} -D_{2,2}^{ - 1} E[G(G^*)']{({D_{2,2}^*}^\prime )^{ - 1}}-(D_{2,2}^*)^{ - 1}E[G^*G']{(D_{2,2}^\prime )^{ - 1}} + {(D_{2,2}^*)^{ - 1}}{Y^*}{(D{_{2,2}^* }')^{ - 1}} \\
	& =D_{2,2}^{ - 1}Y{(D_{2,2}^\prime )^{ - 1}} -{(D_{2,2}^*)^{ - 1}}{Y^*}{(D{_{2,2}^* }')^{ - 1}}-{(D_{2,2}^*)^{ - 1}}{Y^*}{(D{_{2,2}^* }')^{ - 1}}+ {(D_{2,2}^*)^{ - 1}}{Y^*}{(D{_{2,2}^* }')^{ - 1}} \\
	& =D_{2,2}^{ - 1}Y{(D_{2,2}^\prime )^{ - 1}} - {(D_{2,2}^*)^{ - 1}}{Y^*}{(D{_{2,2}^* }')^{ - 1}}=M,
\end{align}
which implies that $M$ is positive semi-definite.\qed

\bigskip
\noindent{\textbf{Proof of Proposition \ref{prop:optW}}}:
	Inserting \eqref{eq: optW} into the definition of $D_{2,1}^*$, we obtain
	\[(D_{2,1}^*)'=nE[(\partial_{\theta_1}V(Y_{t_0},Y_{t_1}))'W^*(Y_{t_0}) ]=n( {Q_1}{C_1}+{Q_2}{C_2} + {Q_3}{C_3} ) \]
	and
	\[S_{1,2}^*=E[\tilde{f}_1V'(Y_{t_0},Y_{t_1})W^*(Y_{t_0}) ]={Q_4}{C_1}+{Q_3'}{C_2} +{Q_5}{C_3} .\]
	Substituting \eqref{eq: optW} into \eqref{eq:H2} for $H_2(y)$, we have
	\begin{align}
		& E[\partial_{\theta_1} V(Y_{t_0},Y_{t_1})|Y_{t_0}=y]{D_{1,1}^{ - 1}}S_{1,1}{(D_{1,1}')^{ - 1}}(n {Q_1}{C_1}+n{Q_2}{C_2} +n {Q_3}{C_3} ) \notag \\
		& - E[\partial_{\theta_1} V(Y_{t_0},Y_{t_1})|Y_{t_0}=y]{D_{1,1}^{ - 1}} ({Q_4}{C_1}+{Q_3'}{C_2} +{Q_5}{C_3})\notag \\
		& - \frac{1}{n}E[V\tilde{f}_1'|Y_{t_0}=y]{(D_{1,1}')^{ - 1}}(n{Q_1}{C_1} +n{Q_2}{C_2}+n{Q_3}{C_3})\notag \\
		& + E[\partial_{\theta_2} V(Y_{t_0},Y_{t_1})|Y_{t_0}=y]{C_1} + E[\partial_{\theta_1} V(Y_{t_0},Y_{t_1})|Y_{t_0}=y]{C_2} + E[V\tilde{f}'_1|Y_{t_0}=y]{C_3}. \notag
	\end{align}
Using $H_2(y)=H_1(y)C$ and \eqref{eq:H1} for $H_1(y)$, we obtain that the coefficient of $E[\partial_{\theta_1}V(Y_{t_0},Y_{t_1})|Y_{t_0}=y]$ and $E[V\tilde{f}_1'|Y_{t_0}=y]$ are equal to 0, and $C_1=C$. Since the optimal weighting matrix is only unique up to scale, we can normalize $C_1$ to $I$, and we have
\[\left\{ {\begin{array}{*{20}{c}}
{nD_{1,1}^{ - 1}{S_{1,1}}{(D_{1,1}')^{ - 1}}({Q_1} + {Q_2}{C_2} + {Q_3}{C_3}) - D_{1,1}^{ - 1}({Q_4} + {Q_3}'{C_2} + {Q_5}{C_3}) + {C_2} = 0},\\
{ - (D_{1,1}')^{ - 1}({Q_1}+{Q_2}{C_2} +{Q_3}{C_3}) + {C_3} = 0.}
\end{array}} \right.\]
Some simplifications give us the desired result. \qed

\bibliographystyle{chicago}
\bibliography{bibfile}

\end{document}